%% file: Folding.tex
\documentclass[12pt]{amsart}
\usepackage[hmargin=3cm,vmargin=3.5cm]{geometry}

\usepackage{xspace,amssymb,amsfonts}
\usepackage{amsthm,amsmath,amscd,stmaryrd}

\usepackage[dvips]{epsfig}
\usepackage{graphicx}
\usepackage[all]{xy}
\SelectTips{cm}{}

\newcommand{\ig}[2]{\vcenter{\xy (0,0)*{\includegraphics[scale=#1]{fig/#2}} \endxy}}

\newcommand{\igc}[2]{\begin{center} \includegraphics[scale=#1]{fig/#2} \end{center}}

\usepackage{tikz}
\usetikzlibrary{calc}

\usepackage{pinlabel}

\RequirePackage{color}
\definecolor{myred}{rgb}{0.75,0,0}
\definecolor{mygreen}{rgb}{0,0.5,0}
\definecolor{myblue}{rgb}{0,0,0.65}

\RequirePackage{ifpdf}
\ifpdf
  \IfFileExists{pdfsync.sty}{\RequirePackage{pdfsync}}{}
  \RequirePackage[pdftex,
   colorlinks = true,
   urlcolor = myblue, 
   citecolor = mygreen, 
   linkcolor = myred, 
   pagebackref,
   bookmarksopen=true]{hyperref}
\else
  \RequirePackage[hypertex]{hyperref}
\fi

\newtheorem{thm}{Theorem}[section]

\newtheorem{prop}[thm]{Proposition}
\newtheorem{cor}[thm]{Corollary}
\newtheorem{claim}[thm]{Claim}

\newtheorem*{prop*}{Proposition}

\theoremstyle{definition}
\newtheorem{defn}[thm]{Definition}

\newtheorem{example}[thm]{Example}

\theoremstyle{remark}
\newtheorem{remark}[thm]{Remark}

\numberwithin{equation}{section}

\usepackage{bbm}

    \def\CM{{\mathbbm{C}}}
    \def\DM{{\mathbb{D}}}

  \def\hg{{\mathfrak h}}

    \def\NM{{\mathbbm{N}}}

    \def\RM{{\mathbbm{R}}}
    \def\SM{{\mathbb{S}}}

    \def\XM{{\mathbb{X}}}
    
    \def\ZM{{\mathbbm{Z}}}
\newcommand{\IM}{{\mathbbm{1}}}

  \def\ab{{\mathbf a}}  
    
    \def\CC{{\mathcal{C}}}
    \def\DC{{\mathcal{D}}}
  \def\eb{{\mathbf e}}

  \def\hb{{\mathbf h}}  \def\HC{{\mathcal{H}}}

    \def\KC{{\mathcal{K}}}

    \def\OC{{\mathcal{O}}}

\def\al{\alpha}
\def\be{\beta}

\def\Ga{\Gamma}
\def\de{\delta}

\def\la{\lambda}

\def\si{\sigma}

\def\ze{\zeta}

\let\phi=\varphi

\renewcommand{\tilde}{\widetilde}

\newcommand{\un}{\underline}
\newcommand{\ov}{\overline}

\newcommand{\ot}{\otimes}
\newcommand{\sqot}{\boxtimes}
\newcommand{\pa}{\partial}
\newcommand{\co}{\colon}
\renewcommand{\to}{\rightarrow}

\newcommand{\sumset}{\stackrel{\scriptstyle{\oplus}}{\scriptstyle{\subset}}}
\newcommand{\define}{\stackrel{\mbox{\scriptsize{def}}}{=}}

\renewcommand{\sl}{\mathfrak{sl}}

\DeclareMathOperator{\Kar}{\bf{Kar}}
\DeclareMathOperator{\Hom}{Hom}

\DeclareMathOperator{\End}{End}

\DeclareMathOperator{\grank}{grank}
\DeclareMathOperator{\gdim}{gdim}
\DeclareMathOperator{\std}{std}
\DeclareMathOperator{\reg}{reg}

\DeclareMathOperator{\Rep}{Rep}
\DeclareMathOperator{\Sym}{Sym}

\DeclareMathOperator{\Tr}{Tr}

\usepackage{tikz-cd}

\newcommand{\Zvv}{\ZM[v,v^{-1}]}
\newcommand{\SBim}{\SM{\rm{Bim}}}
\DeclareMathOperator{\For}{\mathcal{F}or}

\begin{document}

\title{Folding with Soergel bimodules}

\author{Ben Elias} \address{University of Oregon, Eugene}

\begin{abstract} We give a gentle introduction to the concept of folding. That is, we provide an elementary discussion of equivariant categories, their weighted Grothendieck groups, and
the technical aspects of computing with them. We then perform the computations required to confirm that quasi-split Hecke algebras with unequal parameters are categorified by
equivariant Soergel bimodules, in almost every case. \end{abstract}

\maketitle


\section{Introduction}
\label{sec-intro}

Let $\si$ be an automorphism of a Dynkin diagram $\Ga$. Let $W$ be the Weyl group corresponding to $\Ga$, which inherits an action of $\si$. It is an old theorem of Steinberg
\cite{Steinberg} that the $\si$-fixed elements of $W$ themselves form a Coxeter group $W^\si$. Similar statements can be made about the invariant parts of the corresponding Lie algebra or
root system, under the restriction that the orbits of $\si$ contain only mutually-distant vertices in $\Ga$. The most familiar examples of this are the embeddings of $B_{k+1}$ inside
$A_{2k+1}$, and of $C_{k+1}$ inside $D_{k+2}$, where in both examples the Dynkin diagrams of the larger group are visually ``folded'' by an involution $\si$ to produce the Dynkin diagram
of the subgroup. Thus, \emph{folding} is the word chosen to describe the general phenomenon that constructions attached to $W$, when considered $\si$-invariantly or $\si$-equivariantly,
can be related to constructions attached to $W^\si$.

Automorphisms of Dynkin diagrams and the induced automorphisms of Chevalley groups were studied by Chevalley in the 1950s \cite{Chevalley}. Steinberg \cite{Steinberg} introduced the
twisted Chevalley groups which are obtained by taking invariants under these automorphisms. The modern, categorical approach was pioneered by Lusztig in his books
\cite{LuszQuantum,LuszUnequal03}: it studies $\si$-equivariant categories and their decategorifications. This will be the version of folding discussed in this paper.

A primary example of folding is the categorification of quantum groups due to Lusztig \cite{LuszQuantum}. For a simply-laced Dynkin diagram $\Ga$, an appropriate category of perverse
sheaves on the quiver variety corresponding to $\Ga$ (with arbitrary orientation) will categorify the positive half of the quantum group. For non-simply-laced Dynkin diagrams, this
construction will not work. However, every non-simply-laced Dynkin diagram (i.e. types $B$, $C$, $F$, and $G$) can be folded from a simply-laced Dynkin diagram. The automorphism $\si$
lifts to an automorphism of the quiver variety, and one can consider $\si$-equivariant perverse sheaves. This equivariant category will have a Grothendieck group that is too large, but
imposing some \emph{trace relation} on this Grothendieck group will yield the positive half of the quantum group for the folded Dynkin diagram. In this example, folding is crucial as a
means of approaching non-simply-laced Dynkin diagrams.

Another example of folding is the categorification of Hecke algebras, due to Kazhdan-Lusztig (with the folding part due to Lusztig \cite{LuszUnequal03}). For any Weyl group $W$, the
category of perverse sheaves on the flag variety of $W$, stratified with respect to the Bruhat decomposition, will categorify the Hecke algebra $\HC(W)$ with equal parameters. Again, the
automorphism $\si$ lifts to an automorphism of the flag variety. The Grothendieck group of the category of $\si$-equivariant perverse sheaves, after imposing the trace relation, is
isomorphic to the Hecke algebra $\HC(W^\si, L)$ with unequal parameters. Here, $L$ is the \emph{weight function} on $W^\si$ determined by the embedding, which sends $w \in W^{\si}$ to its
length viewed as an element of $W$ (as opposed to its length as an element of the Coxeter group $W^\si$). In this example, non-simply-laced Weyl groups can be approached in two ways, using
their intrinsic flag varieties or using folding, and this categorifies two very different algebras! One should not confuse $\HC(W^\si)$ with $\HC(W^\si, L)$, but should think of them as
being essentially unrelated.

In this paper, we refer to the Grothendieck group of an equivariant category with a trace relation imposed as a \emph{weighted Grothendieck group}.

The main examples of folding have all been geometric, and for this reason it is traditionally limited to Weyl groups and affine Weyl groups, Dynkin diagrams and affine Dynkin diagrams.
However, there is no essential reason for this limitation. Generalizing Steinberg's theorem, Lusztig \cite[Appendix]{LuszUnequal14} and H\'ee \cite{Hee} have shown the following: when a
finite group $G$ acts on a Coxeter group $W$ by automorphisms, preserving the set $S$ of simple reflections, and the simple reflections in the orbit of any $s \in S$ generate a finite
subgroup, then the invariant subgroup $W^G$ is itself a Coxeter group. An embedding of one Coxeter group inside another in this fashion is called a \emph{quasi-split embedding}, and the
corresponding weight function $L$ is called \emph{quasi-split}.

\begin{remark} The literature also occasionally makes assumptions about the action of $G$ which are not truly necessary. It is sometimes assumed that the orbits of $G$ consist of
mutually-distant simple reflections. It is also sometimes assumed that $G$ is cyclic, generated by $\si$; there are geometric reasons that this is reasonable, because one can often
interpret $\si$ as a modified Frobenius automorphism in finite characteristic. Regardless, even when $G$ is not cyclic, a choice of a \emph{weighting element} $\si \in G$ is required to
impose a trace relation on the Grothendieck group. \end{remark}

Meanwhile, Soergel \cite{Soer07} has given an algebraic (not geometric) definition of a graded monoidal category $\SBim(W)$ attached to any Coxeter group $W$, whose objects are now
called \emph{Soergel bimodules}, and he proved under mild restrictions that the Grothendieck group of $\SBim(W)$ is isomorphic to $\HC(W)$. For Weyl groups, with coefficients in a field
of characteristic zero, Soergel proved that $\SBim(W)$ is equivalent to the aforementioned category of perverse sheaves.

Thus it is very natural to ask whether equivariant Soergel bimodules $\SBim(W)_G$ should categorify Hecke algebras with unequal parameters (in the quasi-split case). The goal of this
paper is to give an elementary, computational proof of this fact. We give a precise statement in Theorem \ref{thm:main}.

Continuing work of the author, Khovanov, and Libedinsky \cite{EKho, ECathedral, LibRA} in special cases, the author and Williamson have provided a presentation of $\SBim(W)$ by
generators and relations, using the language of planar diagrammatics, for an arbitrary Coxeter group $W$ \cite{EWGr4sb}. The new mathematical content in this paper is a series of
calculations performed using this diagrammatic calculus, which explicitly computes decompositions of certain tensor products in the equivariant category $\SBim(W)_G$. We hope this paper
serves as an advertisement for the power of diagrammatics, and its ability to effectively address questions even in equivariant categories. Interesting results can be reduced to
elaborate computational exercises.

\begin{remark} \label{dontconfuse} Suppose that $G$ acts on the Coxeter system $(W,S)$. As noted above, one should not confuse $\SBim(W)_G$, the category of equivariant Soergel
bimodules for $W$, with $\SBim(W^G)$. Even though both categories have Grothendieck groups related in some way to $W^G$, the categories are extremely different and wholly unrelated. It
should be possible to derive a diagrammatic calculus for the equivariant category itself, though we have not done this here. \end{remark}

The idea to examine equivariant Soergel bimodules, being as natural as it is, was independently pursued by multiple parties. Since the proof of the Soergel conjecture (in characteristic
zero) by the author and Williamson \cite{EWHodge}, Soergel bimodules have become a more tractable and attractive place to work. Soon after \cite{EWHodge} appeared, Lusztig updated his
book \cite{LuszUnequal14} to account for Soergel bimodules, adding a section on equivariant Soergel bimodules and also giving a proof of the above result. Our approach is quite
different from Lusztig's, being entirely computational in nature, and so we believe our independent proof will still be quite useful. While our proof as stated does rely on the Soergel
conjecture, our computations do not, and thus this paper will be important in future analysis of equivariant Soergel bimodules in situations, such as finite characteristic, where the
Soergel conjecture fails.

\begin{remark} The statement that $\SBim(W)_G$ has weighted Grothendieck group isomorphic to $\HC(W^G,L)$ is false in general. It will be true when $G$ is abelian and the weighting
element $\si \in G$ acts transitively on the $G$-orbits in $S$; under these conditions, we reduce to the case of a cyclic group generated by $\si$ in \S\ref{sec-eq-and-wt}. When $G$ is
abelian and $\si$ is the identity, then the weighted Grothendieck group is isomorphic to $\HC(W)^G$ instead. For general $G$ and $\si \in G$, the weighted Grothendieck group of
$\SBim(W)_G$ is some genuinely new algebra, which should have a number of attractive properties (like cell theory and a canonical basis). It remains to see whether these ``exotic Hecke
algebras" are worth studying. Regardless, the computations performed in this paper will help one to compute the weighted Grothendieck group in general, such as for the action of the
symmetric group $S_3$ on $D_4$. \end{remark}

When $W$ is a Weyl group or affine Weyl group, our main theorem already followed from (highly nontrivial) geometric considerations. Allowing for general Coxeter groups does extend the
generality of the notion of folding slightly. For example, one can now view the Coxeter group of type $A_1$ as the invariants inside the Coxeter group of type $I_2(m)$ (the dihedral
group) under the non-trivial automorphism of $I_2(m)$. However, this is effectively the only new example amongst irreducible, finite Coxeter groups, as the Coxeter systems of type $H_3$
and $H_4$ have no automorphisms. One can also add new examples of folding amongst reducible Coxeter groups, such as folding $k$ copies of $H_3$ into a single copy under an automorphism
which permutes the copies cyclically; such examples are not usually referred to as folding in the literature, but we do treat them here.

A similar question can be posed for quantum groups, where Khovanov-Lauda-Rouquier's quiver Hecke algebras \cite{KhoLau09, Rouq2KM-pp} give an algebraic categorification of the positive
half of the quantum group, and agree with perverse sheaves on quiver varieties in simply-laced type by work of Varagnolo-Vasserot \cite{VarVas}. It should be possible to perform the
analogous computational analysis of the equivariant quiver Hecke category in simply-laced type. McNamara \cite{McNamaraFolding} has a recent paper to this effect, although with
different methods (he exploits the crystal structure on simple objects to prove his main theorem). For non-simply-laced Dynkin diagrams, this would provide two distinct
categorifications of the positive half of the quantum group: the equivariant, geometric, folding categorification which will categorify Lusztig's canonical basis, and the non-geometric
categorification which categorifies a different positive basis. As with Remark \ref{dontconfuse}, these two categories should not be confused.

There are some numerical artifacts which arise freely from categorification results. For example, the Soergel conjecture \cite{EWHodge} demonstrated that the indecomposable objects in
$\SBim(W)$ descend to the Kazhdan-Lusztig basis of $\HC(W)$. One consequence is that the structure constants of $\HC(W)$, with respect to this basis, are non-negative integers, being
the dimensions of certain multiplicity spaces. The dimension is the trace of the identity map of the multiplicity space. Meanwhile, ``multiplicity spaces" in the equivariant category
are representations of a stabilizer group (this is a lie, but philosophically true), and the trace relation on the Grothendieck group implies that structure constant of $\HC(W^G,L)$ are
traces of other elements acting on multiplicity spaces. When $G = \ZM / 2 \ZM$, this statement is a form of \emph{signed positivity}, that a (possibly negative) number is the difference
of two positive numbers, the dimensions of the $+1$ and the $-1$ eigenspace.

There is another ``folding-esque'' scenario, due to Lusztig and Vogan \cite{LusVog12}. Here, one begins with an involution $\si$ of a Coxeter system $(W,S)$, and considers the
involution of $W$ given by $w \mapsto \si(w^{-1})$. It is easy to lift this involution to an anti-monoidal involution on $\SBim(W)$, using adjunction to replace the inverse. In a
follow-up paper, we will apply the same techniques to this equivariant category, and prove that the equivariant Grothendieck group (with a trace relation) is isomorphic to the
Lusztig-Vogan representation. This proves a signed positivity conjecture from \cite{LusVog12}. Just like for this paper, Lusztig and Vogan have already proven this result themselves
\cite{LusVog14} using Soergel bimodules and the results of \cite{EWHodge}, though we hope our direct computational approach will still be illuminating.

We now discuss the contents of this paper. This paper is purposely written for the reader who may be familiar with diagrammatic categorification, but less familiar with equivariant
categories and Coxeter groups.

Chapter \ref{sec-heckebackground} contains an introduction to quasi-split embeddings of Coxeter groups, and to Hecke algebras with unequal parameters. The least well-known aspect of
this introduction is a non-standard presentation of the Hecke algebra with unequal parameters in the quasi-split case, using the Kazhdan-Lusztig generators. The braid relations are
transformed into certain unfamiliar relations, one for each finite dihedral group. These are the relations we categorify in \S\ref{sec-folding}.

Chapter \ref{sec-eq-and-wt} discusses the generalities of equivariant categories and introduces the idea of a \emph{weighted Grothendieck group}. We believe that weighted Grothendieck groups are slightly easier to think about and work with than the trace relation mentioned above. It also allows for more general discussion, beyond the case
of cyclic groups. Thus we have presented the ideas in a different way than is done in Lusztig's books; however, the main ideas are clearly indebted to him. We provide some general
results that help to compute the weighted Grothendieck groups of mixed categories. In particular, for an abelian group $G$ with an element $\si$ which acts transitively on each
$G$-orbit, computing the $\si$-weighted Grothendieck group of a mixed category will reduce to the trace relation of Lusztig. Most of the results in this chapter should be well-known to
the experts, but we were not able to find a comprehensive reference, so we have tried to give an exposition for the apprentice.

Chapter \ref{sec-sbim} gives a terse summary of the category of Soergel bimodules, its diagrammatic presentation due to the author and Williamson, and the Soergel categorification
theorem.

Finally, chapter \ref{sec-folding} performs numerous computations in the category of equivariant Soergel bimodules, for group actions where $W^G$ is dihedral. These computations are
sufficient to categorify the braid-like relations of the Hecke algebra with unequal parameters. Then, using the Soergel conjecture to deduce that $\SBim(W)$ is a mixed category, and
using the abstract results of \S\ref{sec-eq-and-wt}, we deduce our main theorem.

\begin{remark} There is an action of $\ZM / 2 \ZM$ on the Weyl group of type $F_4$, with invariant subgroup $I_2(8)$. We have not even attempted to perform the computations here. This
missing case should follow by analogous computations. Regardless, the result has already been proven by Lusztig. \end{remark}

Again, we do rely heavily on the truth of the Soergel conjecture for our categorification result. However, we do not use any of the Hodge-theoretic machinery required to prove the
Soergel conjecture in \cite{EWHodge}. Moreover, our computations do not rely on the Soergel conjecture, and apply in more generality.

{\bf Acknowledgments.} The author would like to thank George Lusztig and Meinolf Geck for useful conversations, as well as Geordie Williamson, whose mellifluous influence is hardly
superfluous. This research was performed while the author was supported by an NSF postdoctoral fellowship, DMS-1103862.

\section{Hecke algebras and unequal parameters}
\label{sec-heckebackground}

We assume the reader is familiar with the notion of a Coxeter group, as well as the notions of length, reduced expressions, and the Bruhat order. An excellent introduction can be found
in \cite{Humphreys}. We now give some background on weighted Coxeter groups and Hecke algebras with unequal parameters, following Lusztig's book \cite{LuszUnequal03}.

\subsection{Weighted Coxeter groups}
\label{subsec-wtedCox}

Fix a Coxeter system $(W,S)$. For a pair of simple reflections $s,t \in S$, we let $m_{st}$ denote the order of $(st)$. We let $\ell$ denote the length
function. For simplicity, we assume $S$ is finite, though everything can be adapted to the case where $S$ is infinite.

\begin{defn} A \emph{weight function} is a map $L \co W \to \ZM$ such that $L(vw) = L(v) + L(w)$ whenever $\ell(vw) = \ell(v)+\ell(w)$. A Coxeter system paired
with a weight function $(W,S,L)$ is called a \emph{weighted Coxeter system}. \end{defn}

A length function is determined by the values of $L(s)$ for $s \in S$. It is easy to deduce that $L(s) = L(t)$ whenever $m_{st}$ is odd. We say that $L$ is \emph{positive} if
$L(s)>0$ for all $s \in S$, which we assume henceforth.

\begin{example} The length function $\ell$ is also a weight function, and weighted Coxeter systems of the form $(W,S,\ell)$ are called \emph{split}. \end{example}

Non-split weight functions arise naturally when one Coxeter group is embedded inside another in a ``Coxeter-esque'' way.

\begin{defn} Let $(W',S')$ be a Coxeter system. A partition of $S'$ into subsets $I(s)$, parametrized by elements $s$ of a set $S$, is called a \emph{finitary partition} if the parabolic subgroup generated by $I(s)$ is finite for each $s \in S$. We identify the set $S$ with a subset of $W'$, letting $s = w_{I(s)}$, the longest element in the parabolic subgroup for $I(s)$.  \end{defn}

\begin{defn} \label{defn:embedded} Let $(W',S')$ be a Coxeter system, and $I$ a finitary partition parametrized by $S$. Let $W$ be the subgroup of $W'$ generated by $s \in S$. It is
possible that $(W,S)$ is also a Coxeter system, in which case we call this setup a \emph{Coxeter embedding}, and let $\phi_I \co W \to W'$ denote the inclusion map. It equips $(W,S)$
with a positive weight function, called the \emph{embedded weight function}, defined by $L(w) = \ell(\phi_I(w))$. \end{defn}

\begin{defn} \label{defn:stdemb} A Coxeter embedding where, for each $s \in S$, the set $I(s)$ consists only of commuting simple reflections will be called a \emph{standard Coxeter embedding}. In this case, $w_{I(s)} = \prod_{t \in I(s)} t$, and $L(s) = \#I(s)$. \end{defn}

When discussing Coxeter embeddings we will often omit $\phi_I$ from our notation, thus letting $s \in S$ denote both an element in $W$ and its image in $W'$.

\begin{example} \label{dihedral to A} Let $(W',S')$ have type $A_n$, with $S' = \{s_1,s_2,\ldots,s_n\}$. Let $(W,S)$ have type $I_2(n+1)$, a dihedral group with $S = \{t,u\}$ and
$m_{tu}=n+1$. Set $I(t) = \{s_i\}_{i \text{ even}}$ and $I(u) = \{s_i\}_{i \text{ odd}}$. This yields a standard Coxeter embedding. This embedding is related to the 2-sided cell in $I_2(n+1)$ consisting of elements with a unique reduced expression, as shown by Lusztig \cite{LusSqint}. \end{example}

\igc{1}{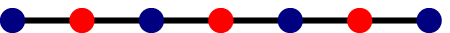}

Sometimes the partition of $S'$ into subsets comes from the orbits of a group action. As the example above shows, this is not always the case.

\begin{defn} \label{defn:quasisplit} Let $G$ be a finite group acting on a Coxeter system $(W', S')$. In other words, $G$ acts on the set $S'$, with $m_{st} = m_{g(s) g(t)}$. A Coxeter
embedding where $S$ is in bijection with the orbit space $S'/G$, and $I(s)$ is the corresponding orbit, will be called \emph{quasi-split}. In this case, $W = (W')^G$ is the subgroup of
$G$-invariants. If we want to emphasize the group, we will call the embedding \emph{$G$-split}. \end{defn}

It is known that every finitary partition of $S'$ which comes from the orbits of a group action gives rise to a Coxeter embedding. See \cite{GecIan} for an elementary proof. However,
we do not know whether a general finitary partition always gives rise to a Coxeter embedding. We were unable to find this result, or a counterexample, in the literature.

There is a classification of quasi-split embeddings where the Coxeter groups involved are finite. The table on \cite[p227]{GecJac} lists the possibilities when $W'$ is a connected
Weyl group. Allowing for connected non-crystallographic finite Coxeter groups only adds one family of examples (Example \ref{dihedraltoA1} below).

\begin{example} \label{B to A} The group $\ZM/2\ZM$ acts on the Coxeter system of type $A_{2k+1}$, and the invariants have type $B_{k+1}$. This is a standard quasi-split embedding.
If the simple reflections in type $B_{k+1}$ are labeled $t_i$ for $1 \le i \le k$ and $u$, where $\{t_i\}$ generates a subgroup of type $A_{k}$, then the embedded weight function
satisfies $L(t_i)=2$ and $L(u)=1$.

\igc{1}{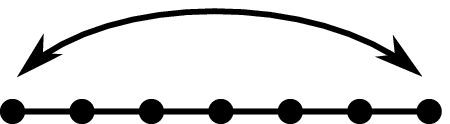}
\end{example}

\begin{example} \label{B to A other} The group $\ZM/2\ZM$ acts on the Coxeter system of type $A_{2k+2}$, and the invariants have type $B_{k+1}$. This is a non-standard quasi-split embedding, as the middle orbit has type $A_2$ instead of $A_1 \times A_1$. In this example, $L(t_i)=2$ and $L(u) = 3$. \end{example}

\begin{example} \label{B to D} The group $\ZM/2\ZM$ acts on the Coxeter system of type $D_{k+2}$, and the invariants have type $B_{k+1}$. This is a standard quasi-split Coxeter embedding, with $L(t_i)=1$ and $L(u)=2$.

\igc{1}{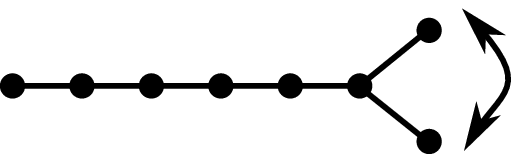} \end{example}

\begin{example} \label{E to F} The group $\ZM / 2 \ZM$ acts on $E_6$ with invariant subgroup $F_4$. \end{example}

\begin{example} \label{star shaped} Let $(W',S')$ be a star-shaped Coxeter group. That is, $S' = \{s_1,s_2, \ldots, s_n, v\}$ where $m_{s_i s_j} = 2$ and $m_{s_i v}=3$ for all $i,j$.
The symmetric group $S_n$ or its cyclic subgroup $\ZM / n \ZM$ will permute the spokes of the star, giving a standard quasi-split embedding whose invariant subgroup will be dihedral. If
$t$ corresponds to the vertex and $u$ corresponds to the orbit containing the spokes, then $L(t)=1$ and $L(u)=n$. Moreover, $m_{tu}=3$ when $n=1$, $m_{tu}=4$ when $n=2$, $m_{tu}=6$ when
$n=3$, and $m_{tu}=\infty$ when $n \ge 4$. The embedding of $G_2$ inside $D_4$ is the special case $n=3$.

\igc{1}{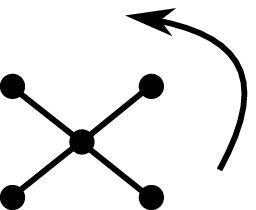} \end{example}

\begin{example} \label{dihedraltoA1} The group $\ZM / 2 \ZM$ acts on the dihedral group $I_2(m)$ for $m < \infty$, with invariant group of type $A_1$. \end{example}

\begin{example} \label{F to B} The group $\ZM / 2 \ZM$ acts on $F_4$ with quotient $I_2(8)$. \end{example}

Let us list some other examples of quasi-split embeddings.

\begin{example} \label{affine inclusions} There are many examples of inclusions of affine Weyl groups, analogous to some of the embeddings of finite Weyl
groups above. \end{example}

\begin{example} \label{many copies} Let $(W,S)$ be a Coxeter group, and let $(W',S')$ be the $k$-fold disjoint multiple of $W$. In order words, $S' = S \times
\{1,\ldots,k\}$, with $m_{(s,i) (t,i)} = m_{st}$ and $m_{(s,i) (t,j)} = 2$ when $1 \le i \ne j \le k$. Let $I(s) = \{(s,i)\}_{1 \le i \le k}$. Clearly this
gives a (diagonal) standard quasi-split embedding of Coxeter groups, with $L(s)=k$. \end{example}

\begin{example} \label{affine shrinking} Suppose $k \ge 3$. The Dynkin diagram of the affine Weyl group $\tilde{A}_{nk-1}$ looks like a circle with $(nk)$ nodes. It has an action of $\ZM / n \ZM$, with
invariant subgroup of type $\tilde{A}_{k-1}$. \end{example}

\begin{example} \label{universal} Let $(W,S)$ be a universal Coxeter group (that is, $m_{st}=\infty$ for all $s,t \in S$). One can obtain any positive weight function $L$ as an embedded
weight function for a standard quasi-split embedding, in many ways. A simple such realization is to let $S'$ be a complete multipartite graph, where an edge means $m_{uv}=\infty$, and
the lack of an edge means $m_{uv}=2$. \end{example}

For fun and edification, we list some non-quasi-split embeddings.

\begin{example} \label{A1 to any} There is an embedding of $A_1$ into any finite Coxeter group, sending the generator to the longest element. This is quasi-split only for ($k$-fold
products of) dihedral groups. \end{example}

\begin{example} \label{G2 to B3} There is an embedding of $G_2$ into $B_3$. It is standard but not quasi-split. \end{example}

\begin{example} \label{H4 to E8} There is a standard, non-quasi-split embedding of $H_4$ into $E_8$, where $L(s)=2$ for all simple reflections. There is a similar embedding of $H_3$
into $D_6$. These were shown to me by Lusztig, and like Example \ref{dihedral to A} above are associated to the 2-sided cell of elements with a unique reduced expression in $H_3$ and $H_4$. See \cite[Remark 3.9]{LusSqint}. \end{example}

\begin{example} \label{B to A weird} Embed $B_2$ into $A_n$ by letting $I(t) = \{s_1, s_n\}$ and $I(u) = \{s_i\}_{i=2}^{n-1}$. This is a non-standard, non-quasi-split Coxeter embedding,
for which the longest element of $B_2$ is not sent to the longest element of $A_n$. \end{example}.

\begin{remark} Categorifying Hecke algebras with unequal parameters for non-quasi-split embeddings is a major open question. Examples \ref{dihedral to A} and \ref{H4 to E8} are
particularly important, because they might provide an explanation for the apparent ``geometric'' behavior of non-crystallographic Coxeter groups like $I_2(m)$ and $H_4$. \end{remark}

\subsection{Quasi-split embeddings and cyclic groups}
\label{subsec-quasisplit}

When quasi-split embeddings are discussed in the literature, it is sometimes implicitly assumed that the group $G$ which acts is cyclic, generated by an element $\si$. This is not much
of an assumption, for the following reason. The embedded group $(W,S,L)$ only depends on the orbits of the action, and not on the group action itself, and every quasi-split embedding
can be realized for a cyclic group. This is evidenced by the comprehensive list of examples above. The reason that cyclic groups are important is that the original definition of
quasi-split embeddings came from embeddings of algebraic groups in finite characteristic, and $\si$ arises from the Frobenius automorphism. See \cite[\S 16]{LuszUnequal03} for more details.

Later in this paper we will define a category $\SBim_G$ attached to any Coxeter system $(W', S')$ with a group action by a group $G$, and a weighted equivariant Grothendieck group
corresponding to an element $\si \in G$. Now, it is not just the orbits, but the group itself which plays a significant role. One hopes that this weighted Grothendieck group is
isomorphic to the Hecke algebra with unequal parameters corresponding to the invariant subgroup for the Coxeter embedding (to be defined soon). We will prove in
\S\ref{subsec-quadraticII} that this is only possible when $\si$ acts transitively on each orbit, so that the orbits of the cyclic subgroup $\langle \si \rangle$ agree with the orbits
of $G$. This is one reason why one may wish to assume that $G$ is cyclic.

\begin{example} There is an embedding of $B_2$ into $A_3 \times A_3$, with $L(s)=2$ and $L(t)=4$. These orbits could arise from an action of $\ZM/4\ZM$ or an action of $\ZM/2\ZM \times
\ZM/2 \ZM$; this latter group does not contain any element of order 4, so no cyclic subgroup will have the same orbits. \end{example}

However, even when an element $\si \in G$ exists which acts transitively on each orbit, the $G$-equivariant $\si$-weighted Grothendieck group and the $\langle \si \rangle$-equivariant
$\si$-weighted Grothendieck group need not be equal.

The equivariant Grothendieck groups attached to actions of non-cyclic groups will not be Hecke algebras in the traditional sense. However, they may still be algebras with similarly
attractive features, such as an interesting cell theory. It remains to be seen whether these algebras merit further investigation.

\subsection{Dihedral quasi-split embeddings}
\label{subsec-quasisplit2}

Let us make a quick note about the possibilities for finite dihedral subgroups in quasi-split embeddings, which is also evident from the classification above.

\begin{prop} \label{dihedral possibilities} Suppose one has a quasi-split embedding $W \to W'$, and $s,t \in S$ with $m_{st}$ finite. The following list encapsulates all possibilities
for $I(s) \cup I(t)$. \begin{itemize}
\item $m_{st}$ is arbitrary, and $I(s) \cup I(t)$ is a disjoint union of $k$ copies of $I_2(m_{st})$. One has $L(s) = L(t) = k$.
\item $m_{st}=2$, and $I(s) \cup I(t)$ is a disjoint union of $k+l$ copies of $A_1$. One has $L(s) = k$ and $L(t) = l$.
\item $m_{st}=4$, and $I(s) \cup I(t)$ is a disjoint union of $k$ copies of $A_3$. One has $L(s)=k$ and $L(t)=2k$.
\item $m_{st}=4$, and $I(s) \cup I(t)$ is a disjoint union of $k$ copies of $A_4$. One has $L(s)=3k$ and $L(t) = 2k$.
\item $m_{st}=6$, and $I(s) \cup I(t)$ is a disjoint union of $k$ copies of $D_4$. One has $L(s)=k$ and $L(t)=3k$.
\item $m_{st}=8$, and $I(s) \cup I(t)$ is a disjoint union of $k$ copies of $F_4$. One has $L(s)=4k$ and $L(t) = 2k$.
 \end{itemize} In particular, $L(s) = L(t)$ unless $m_{st}=2,4,6,8$. \end{prop}

\begin{proof} Without loss of generality, $W$ is dihedral, and $S' = I(s) \cup I(t)$. It is easy to classify the finite Coxeter systems $(W', S')$ with a group action having only two orbits, and they correspond to those listed above. It remains to show that $W'$ must be finite in order that $m_{st} < \infty$.

When the embedding is standard, the Coxeter element $h = st$ of $W$ is sent under the embedding to a
Coxeter element $h'$ of $W'$. Since the order of $h$ is finite, so is the order of $h'$. Howlett \cite{Howlett} has shown that the order of a Coxeter element
of an infinite group is infinite; therefore, $W'$ must be finite.

In general, one can make the following argument, explained to me by Meinolf Geck. The longest element of $W$ goes to some element $w \in (W')^G$. We claim that $w$ has every simple
reflection in its (left) descent set, which forces it to be a longest element of $W'$. If some $u \in I(s)$ is not in the left descent set of $w$, then neither is any element in the
entire orbit $I(s)$. Standard observations about cosets of parabolic subgroups indicate that $w$ must be a minimal coset representative for the parabolic subgroup generated by $I(s)$,
and thus $sw > w$. But then $sw$ is a longer element in $W$, a contradiction. See \cite[Lemma 2]{GecIan} for a similar proof. \end{proof}

\subsection{Hecke algebras with unequal parameters}
\label{subsec-Heckeeq}

The \emph{Hecke algebra (with unequal parameters)} $\HC = \HC(W,S,L)$ is a $\Zvv$-algebra with \emph{standard generators} $H_s$ for $s \in S$, satisfying relations
\begin{subequations} \label{eqs:hecke} \begin{equation} \label{eq:quadratic} (H_s + v^{L(s)})(H_s - v^{-L(s)}) = 0, \end{equation} \begin{equation} \label{eq:braid} H_s H_t \cdots = H_t
H_s \cdots \end{equation} \end{subequations} where both terms in \eqref{eq:braid} are alternating of length $m_{s,t}$. Relation \eqref{eq:braid} is omitted when $m_{s,t}=\infty$. We assume that $L$ is a positive weight function henceforth.

An \emph{expression} $\un{x} = s_1 s_2 \ldots s_d$ is a sequence of elements $s_i \in S$, and is denoted with an underlined letter. Removing the underline, one obtains the product $x \in
W$. An expression as above is reduced if $\ell(x) = d$. We write $H_{\un{x}} \define H_{s_1} \cdots H_{s_d}$.  When $\un{x}$ is reduced, let $H_x = H_{\un{x}}$, a product which does not
depend on the chosen reduced expression. By convention $H_1 = 1$. The collection $\{H_x\}_{x \in W}$ is the \emph{standard basis} of $\HC$ over the ring $\Zvv$.

Given a polynomial $P \in \Zvv$ we write $P^k$ for the coefficient of $v^k$ in $P$. We write $\HC_{\le n}$ for those linear combinations $\sum_{w \in W} P_w H_w$ where $P_w^k=0$ for
$k>n$. We define $\HC_{\ge n}$, $\HC_{< n}$ and $\HC_{> n}$ in the same way.

The \emph{bar involution} is the $\ZM$-linear map $\ov{(\cdot)} \co \HC \to \HC$ satisfying $\ov{v}=v^{-1}$ and $\ov{H_x}=H_{x^{-1}}^{-1}$. It is an algebra involution. An element $b \in
\HC$ satisfying $\ov{b}=b$ is called \emph{self-dual}. For each $x \in W$ there is a unique element $b_x \in \HC_{\ge 0}$ which is self-dual, and for which $b_x = H_x$ modulo $\HC_{>0}$.
Together, these elements $\{b_x\}_{x \in W}$ are called the \emph{Kazhdan-Lusztig basis} or \emph{KL basis}. The \emph{KL polynomials} $P_{w,x}$ are defined by $b_x = \sum_w P_{w,x} H_w$.
One can show that $P_{w,x}=0$ unless $w \le x$, and that $P_{x,x}=1$.

Our notation for the Hecke algebra differs from Lusztig's. To compare this definition with the definition in \cite{LuszUnequal03}, one must switch $v$ and $v^{-1}$. Then our $H_x$ agrees with
Lusztig's $T_x$, and our $b_x$ agrees with his $c_x$.

In the split case $(W,S,\ell)$, there is a \emph{standard pairing} on $\HC$ with a variety of nice properties. This is a map $\HC \times \HC \to \Zvv$ satisfying
\begin{subequations}
\begin{equation} (b_w x,y) = (x,b_{w^{-1}} y) \end{equation}
\begin{equation} (x b_w,y) = (x,y b_{w^{-1}}) \end{equation}
\begin{equation} \label{gradedorthonormal} (b_w, b_v) = \delta_{wv} + v\ZM[v] \end{equation}
\end{subequations}
for all $v,w \in W$ and all $x,y \in \HC$. The property \eqref{gradedorthonormal} states that the KL basis is \emph{graded orthonormal}.

\subsection{Another presentation}
\label{subsec-KLpresentation}

One special KL basis element is $b_s = H_s + v^{L(s)}$, which we call a \emph{KL generator}. Just as the standard generators generate the Hecke algebra over $\Zvv$, so too do the KL
generators. The presentation for $\HC$ in terms of the KL generators is less well-known; it does not appear to be in the literature. 

It is easy to observe that the quadratic relation on standard generators \eqref{eq:quadratic} is equivalent to the following quadratic relation. \begin{equation} \label{eq:bb} b_s b_s
= (v^{L(s)} + v^{-L(s)}) b_s.\end{equation}

Given a sequence $\un{x}$, we write $b_{\un{x}}$ for the product $b_{s_1} b_{s_2} \cdots b_{s_d}$. Note that typically $b_{\un{x}} \ne b_x$, even for reduced expressions. However,
$b_{\un{x}}$ is self-dual, and one can deduce that for a reduced expression $\un{x}$, \begin{equation} b_{\un{x}} = b_x + \sum_{y < x} Q_{y,\un{x}} b_y \end{equation} for some
polynomials $Q_{y, \un{x}} \in \ZM[(v+v^{-1})]$. For the moment, we will call these terms $Q_{y,\un{x}} b_y$ by the name ``lower terms."

Pick $s,t \in S$ with $m_{st} < \infty$. Let $w_0$ denote the longest element of the dihedral parabolic subgroup $W_{s,t}$ generated by $s$ and $t$. It has two reduced expressions, and
therefore one can find two separate equations for $b_{w_0}$, one as $b_s b_t b_s \ldots$ plus certain lower terms, and one as $b_t b_s b_t \ldots$ plus certain other lower terms. The
equality of these two expressions is an algebraic relation on the generators $b_s$ and $b_t$, and this relation will be equivalent to the braid relation on standard generators
\eqref{eq:braid}.

Therefore, the Hecke algebra has a presentation with KL generators, the relation \eqref{eq:bb}, and some dihedral relation for each pair $s,t$ with $m_{st}<\infty$, given by equating two
expressions for $b_{w_0}$. It remains to find these two expressions for $b_{w_0}$, which we do here for the quasi-split case (i.e. using the restrictions on $L$ coming from Proposition
\ref{dihedral possibilities}).

\begin{prop} \label{prop: quasi-split} Suppose that $(W,S,L)$ is quasi-split, and let $s,t \in S$ with $m_{st} < \infty$. We give two descriptions of $b_{w_0}$, where $w_0$ is the
longest element of $W_{s,t}$. Equivalently, given either reduced expression $\un{x}$ for $w_0 \in W_{s,t}$, we decompose the product $b_{\un{x}}$ into the KL basis.
\begin{itemize}

\item Suppose that $L(s)=L(t)$, and $m_{st}$ is arbitrary. For $1 \le k \le m$ let $\un{k}_s$ denote the reduced expression which alternates between $s$ and $t$, has length $k$, and
ends in $s$. Then \begin{equation} \label{eq:dihedral equal} b_{w_0} = b_{\un{m}_s} + \sum_{k < m} \rho_{k,m} b_{\un{k}_s}. \end{equation} In this equation, the coefficients
$\rho_{k,m}$ are all integers, and the same integers give a relation $$V_{m+1} \cong V^{\ot (m+1)} + \sum_{k < m} \rho_{k,m} V^{\ot (k+1)}$$ in the Grothendieck group of $\mathfrak{sl}_2$
representations. The same equation holds with $s$ and $t$ switched. Equivalently, one has \begin{equation} \label{eq:dihedral equal alt} b_{\un{m}_s} = b_{w_0} + \sum_{k < m}
\kappa_{k,m} b_{k_s}. \end{equation} The positive integers $\kappa_{k,m}$ satisfy $$V^{\ot (m+1)} \cong V_{m+1} \oplus \bigoplus_{k < m} V_{k+1}^{\oplus \kappa_{k,m}}.$$

\item Suppose that $m_{st}=2$, $L(s)=k$ and $L(t)=l$. Then \begin{equation} \label{eq:dihedral m=2} b_{w_0} = b_s b_t = b_t b_s. \end{equation}

\item Suppose that $m_{st}=4$, $L(s)=k$ and $L(t)=2k$. Then \begin{equation} \label{eq:dihedral m=4} b_{w_0} = b_s b_t b_s b_t - (v^k + v^{-k}) b_s b_t = b_t b_s b_t b_s - (v^k + v^{-k}) b_t b_s. \end{equation} Equivalently, one has \begin{equation} \label{eq:dihedral m=4 alt} b_s b_t b_s b_t = b_{stst} + (v^k + v^{-k}) b_{st}. \end{equation} In fact, the same computation holds for $L(s)=3k$ and $L(t)=2k$.

\item Suppose that $m_{st}=6$, $L(s)=k$ and $L(t)=3k$. Then \begin{eqnarray} \label{eq:dihedral m=6} b_{w_0} & = & b_s b_t b_s b_t b_s b_t - (2v^{2k} + 2v^{-2k}) b_s b_t b_s b_t + (v^{4k} + 1 + v^{-4k}) b_s b_t \\ \nonumber & = & b_t b_s b_t b_s b_t b_s - (2v^{2k} + 2v^{-2k}) b_t b_s b_t b_s + (v^{4k}+1+v^{-4k}) b_t b_s.\end{eqnarray} Equivalently, one has \begin{equation} \label{eq:dihedral m=6 alt} b_s b_t b_s b_t b_s b_t = b_{ststst} + (2v^{2k}+2v^{-2k}) b_{stst} + (v^{-4k} + 3 + v^{4k}) b_{st}. \end{equation}

\item The case when $m_{st}=8$ is not difficult, but we set a precedent here and ignore it. \end{itemize} \end{prop}

The proof of the case when $L(s)=L(t)=1$ is discussed at length in \cite{ECathedral}, and implies the case $L(s) = L(t) = k$. The remaining cases are just straightforward computations in
the dihedral group. which we leave as an exercise to the reader. For practice computing with dihedral groups in the unequal parameter case, see chapter 7 of \cite{LuszUnequal03}.

The general method is to find a recursion formula for multiplying a KL basis element with a KL generator. For equal parameters, this recursion is exemplified by the formula \[ b_s
b_{tststs} = b_{stststs} + b_{ststs}. \] The ``closed form'' of \eqref{eq:dihedral equal alt} comes from the observation that this recursion agrees with plethyism formulas for
representations of $\sl_2$, where $V_1 \ot V_k \cong V_{k+1} \oplus V_{k-1}$ for $k \ge 1$. This agreement is actually a shadow of the geometric Satake equivalence. See \cite{ECathedral}
for further explanation.

Meanwhile, the recursion formulas for arbitrary unequal parameters are not difficult to derive. For example, when $0 < L(s) < L(t)$ one has the exemplary formulas \[ b_s b_{tststs} =
b_{stststs} \] and \[ b_t b_{stststs} = b_{tstststs} + (v^{L(t)-L(s)} + v^{L(s) - L(t)}) b_{tststs} + b_{tsts}. \] A ``closed form" formula for $b_{w_0}$ for arbitrary dihedral groups
would require some elementary combinatorics. However, it would be far more interesting to find a category with similar plethyism rules.

\subsection{Invariant subalgebras of Hecke algebras}
\label{subsec-invariants}

In this chapter, we fix a $G$-split embedding $(W,S,L) \to (W',S',\ell)$. One mystery this paper seeks to unravel is the connection between $\HC(W,S,L)$ and
$\HC(W',S',\ell)$. For an element $x \in W$, we write $c_x$ for the KL basis element in $\HC(W,S,L)$. For an element $y \in W'$ (resp. $y \in W$) we write $b_y$ for the KL basis element
of $y$ (resp. of $\phi_I(y)$) in $\HC(W',S',\ell)$.

The group $G$ acts on $W'$ and on $\HC(W',S',\ell)$. Though one has $W = (W')^G \subset W'$, the size of the subalgebra $\HC(W',S',\ell)^G$ is much larger than the size of $W$. After
all, $\HC(W',S',\ell)^G$ has a basis parametrized by orbits of $G$ on $W'$, given by sums $\sum_y b_y$ over the elements $y$ of each orbit. Considering only the singleton orbits, one
obtains a sub-basis $\{b_x\}_{x \in (W')^G}$, though the span of this sub-basis is not closed under multiplication.

We now give examples which compare, for various $x \in W$, the multiplication of $b_x$ in $\HC(W',S',\ell)^G$ with the multiplication of $c_x$ in $\HC(W,S,L)$. These examples will
cover (almost) all the cases in Proposition \ref{prop: quasi-split}.

\begin{example} \label{compare A1A1} Let $W'$ have type $A_1 \times A_1$, with $G$ acting transitively on $S' = \{t,u\}$. Then $s=tu$ generates $(W')^G = W$. One has
\begin{equation} b_s b_s = b_t b_u b_t b_u = (v+v^{-1})^2 b_t b_u = (v^2 + {\color{red}2} + v^{-2}) b_s. \end{equation} In $\HC(W,S,L)$, however, \eqref{eq:bb} implies that \begin{equation} c_s
c_s = (v^2 + v^{-2}) c_s. \end{equation}

More generally, let $W'$ be any finite Coxeter group, with $A_1$ embedded by sending $s$ to the longest element of $W'$. Then one has $b_s^2 = [W'] b_s$ while $c_s^2 = (v^{L(s)} +
v^{-L(s)}) c_s$. Here, $[W']$ denotes the balanced Poincar\'e polynomial of $W'$, which has highest degree term $v^{L(s)}$ and lowest degree term $v^{-L(s)}$, but has various other
terms in between. \end{example}

\begin{example} \label{compare A1^k A1^l} Let $W'$ have type $A_1^{\times k} \times A_1^{\times l}$, with $G$ acting transitively on the first $k$ factors and the last $l$ factors.
Let $s$ denote the product of the first $k$ simple reflections, and $t$ denote the product of the last $l$. Then \begin{equation} b_s b_t = b_{st}, \end{equation} just as
\begin{equation} c_s c_t = c_{st}. \end{equation} \end{example}

\begin{example} \label{compare I2m^k} Let $W'$ have type $I_2(m)^{\times k}$, with $G$ acting transitively on these copies of $I_2(m)$. Let $s$ and $t$ be the generators
of $W$, and let $w_0$ denote the longest element of $W'$, also the longest element of $W$. Then \begin{equation} b_{\un{m}_s} = b_{w_0} + \sum_{k < m} \kappa_{k,m} b_{k_s},
\end{equation} just as \begin{equation} c_{\un{m}_s} = c_{w_0} + \sum_{k < m} \kappa_{k,m} c_{k_s}. \end{equation} \end{example}

\begin{example} \label{compare A3B2} Let $W'$ have type $A_3$, with $G = \ZM / 2 \ZM$ acting on $S' = \{x,y,z\}$ to switch $x$ and $z$. Then $s = y$ and $t = xz$ generate $(W')^G = W$. A
computation yields \begin{equation} b_s b_t b_s b_t = b_y b_x b_z b_y b_x b_z = b_{yxzyxz} + b_{yxyz} + b_{yzyx} + (v+v^{-1}) b_{yxz} = b_{stst} + {\color{blue} (b_{yxyz} + b_{yzyx})}
+ (v+v^{-1})b_{st}. \end{equation} In $\HC(W,S,L)$, however, \eqref{eq:dihedral m=4 alt} implies that \begin{equation} c_s c_t c_s c_t = c_{stst} + (v+v^{-1}) c_{st}. \end{equation}

The same calculation suffices when $W'$ has type $A_3^{\times k}$, and $G$ also permutes these $k$ factors. However, one must replace $v$ with $v^k$. This is a general principle.
\end{example}

\begin{example} \label{compareA4B2} Let $W'$ have type $A_4$, with $G = \ZM/2\ZM$ acting on $S' = \{s_1, s_2, s_3, s_4\}$ in the usual way. Then $t = s_1 s_4$ and $u = s_2 s_3 s_2$ generate $(W')^G = W$ of type $B_2$. Let $w_{123}$ and $w_{234}$ be the longest elements of the corresponding parabolic subgroups of type $A_3$. A computation yields
\begin{equation} b_t b_u b_t b_u = b_{tutu} + {\color{blue} (b_{s_2 s_1 w_{234}} + b_{s_3 s_4 w_{123}})} + {\color{blue} (v+v^{-1}) (b_{s_1 w_{234}} + b_{s_4 w_{123}})} + (v+v^{-1}) b_{tu}. \end{equation} In $\HC(W,S,L)$, however, \eqref{eq:dihedral m=4 alt} implies that \begin{equation} c_t c_u c_t c_u = c_{tutu} + (v+v^{-1}) c_{tu}. \end{equation} \end{example}

\begin{example} \label{compare D4G2} Let $W'$ have type $D_4$, where $S' = \{u_1,u_2,u_3,v\}$, and let $G$ act transitively on the vertices $u_i$. Let $s = v$ and $t= u_1 u_2 u_3$
denote the generators of $W$. Let $w_0 = ststst$ denote the longest element in $W'$ and $W$. A computation yields \begin{align} b_s b_t b_s b_t b_s b_t = b_{ststst} + {\color{blue}
(b_{u_1 u_2 stst} + b_{u_1 u_3 stst} + b_{u_2 u_3 stst})} + {\color{blue} (v+v^{-1}) (b_{u_1 stst} + b_{u_2 stst} + b_{u_3 stst})} \\ \nonumber + (2v^2 + {\color{red} 6} + 2v^{-2})
b_{stst} + {\color{blue} (v+v^{-1}) (b_{vu_1 u_2 st} + b_{vu_1 u_3 st} + b_{v u_2 u_3 st})}\\ \nonumber + (v^{4} + {\color{red} 6v^2} + (3+{\color{red} 9}) + {\color{red} 6v^{-2}} +
v^{-4}) b_{st}. \end{align} In $\HC(W,S,L)$, however, \eqref{eq:dihedral m=6 alt} implies that \begin{equation} c_s c_t c_s c_t c_s c_t = c_{ststst} + (2v^{2}+2v^{-2}) c_{stst} +
(v^{-4} + 3 + v^{4}) c_{st}. \end{equation} \end{example}

We continue our precedent and ignore the case where $W'$ has type $F_4$.

As evidenced by the examples above, mutliplication of the sub-basis $\{b_x\}$ for $x \in W \subset W'$ seems to behave approximately like the multiplication of $\{c_x\}$ in
$\HC(W,S,L)$, although with {\color{red}e\color{blue}x\color{red}t\color{blue}r\color{red}a\color{blue}n\color{red}e\color{blue}o\color{red}u\color{blue}s} terms. Some of these
extraneous terms have the form ${\color{blue} \sum_y b_y}$ over a non-trivial orbit of $G$ in $W'$, while others are multiples ${\color{red} n b_x}$ for $x \in W$. If only there were
some way to eliminate these extraneous terms...

\subsection{Some motivation}
\label{subsec-somemotivation}

Our goal in the remainder of this paper is to give a categorical explanation for the similarities and differences between these formulas. Let us briefly motivate this, as in the
introduction. One important result in the split case is called \emph{KL positivity}.

\begin{thm} When $L=\ell$ (i.e. in the split case), the following properties hold. \begin{itemize} \item $P_{w,x} \in \NM[v]$. \item $b_x b_y = \sum_z m^z_{x,y} b_z$, and $m^z_{x,y}
\in \NM[v,v^{-1}]$. \end{itemize} \end{thm}

This theorem was proven using categorification-theoretic methods by Kazhdan and Lusztig for Weyl groups, and by the author and Williamson for general Coxeter groups \cite{EWHodge}. One
proves that each coefficient of these polynomials is in fact the dimension of a certain vector space. On the other hand, KL positivity as stated will fail for Hecke algebras with unequal
parameters. Various KL polynomials $P_{w,x}$ and structure coefficients $m^z_{x,y}$ have negative coefficients.

In this paper, for a $G$-split embedding, we construct a category $\SBim_G$ from which one can extract vector spaces equipped with an action of $G$. Taking the dimension of these vector
spaces, one can obtain the (positive) coefficients for the subalgebra $\HC(W',S',\ell)^G$. The dimension is the trace of the identity element in $G$. If instead one takes the trace of a
different element of $G$ in special cases, one can obtain the coefficients for the algebra $\HC(W,S,L)$. Traces of other elements need not be positive integers. In this sense does one
cancel out the extraneous terms above: they correspond to sub-vector spaces with zero trace. In other words, the monoidal structure on $\SBim_G$ categorifies the formulas involving
$b_s$ above, while some trace cancellation reduces to the formulas involving $c_s$. The precise way in which $\SBim_G$ is said to categorify $\HC(W,S,L)$ involves weighted Grothendieck
groups, and will be elaborated in \S\ref{sec-eq-and-wt}.

\section{Equivariant categories and weighted Grothendieck groups}
\label{sec-eq-and-wt}

This section contains general background information on equivariant categories and the weighted Grothendieck groups associated to them. Definitions like these were made by Lusztig in
his description of folding for quantum groups \cite{LuszQuantum}. The results of this chapter are entirely folklore, though we were unable to find an exposition with the same level of generality.

\subsection{Definitions}
\label{subsec-defns-eq-and-wt}

\begin{defn} A \emph{strict action} of a group $G$ on a (graded) category $\CC$ is the following data. For each $g \in G$ one has an autoequivalence of $\CC$, also denoted $g$, such
that the identity $e \in G$ is assigned the identity functor. When $\CC$ is graded, these autoequivalences must commute naturally with the grading shift. For each $g,h \in G$, there is
an invertible natural transformation $a_{g,h} \co g \circ h \to gh$, satisfying an \emph{associativity} condition for each triple $g,h,k \in G$: \begin{equation} a_{gh,k} \circ
(a_{g,h} \ot 1) = a_{g,hk} \circ (1 \ot a_{h,k}) \co g \circ h \circ k \to ghk. \end{equation} \end{defn}

\begin{remark} Another definition of a strict action often found in the literature will allow the identity $e \in G$ to be assigned to a functor naturally isomorphic to the identity
functor. One fixes this natural isomorphism, and places additional conditions upon it. It is hardly less general to assert that $e$ is sent to the identity functor.
\end{remark}

We will often abuse notation and identify the functors $g \circ h$ and $gh$, leaving the natural transformation $a_{g,h}$ understood. The associativity condition allows us to identify the
functors $g \circ h \circ k$ and $ghk$ canonically.

\begin{defn} \label{defn:wtdgrth} Let $\DC$ be an additive (graded) category with a strict action of a group $K$, and let $\eta \in K^*$. The \emph{$\eta$-weighted Grothendieck group}
$[\DC]_\eta$ of $\DC$ is the free $\CM$-module generated by the symbols $[M]$ for objects $M \in \DC$, modulo the usual relation that $[M \oplus N] = [M]+[N]$, and the new relation
that $[kM] = \eta(k)[M]$ for $k \in K$. When $\DC$ is graded, this is also a $\CM[v^{\pm 1}]$-module via $v[M] = [M(1)]$. \end{defn}

\begin{remark} One can generalize this definition to define a $V$-weighted Grothendieck group, for any representation $V$ of $K$. A typical element is $[M] \ot v$ for $v \in V$, and the
new relation is $[kM] \ot v = [M] \ot k^{-1}v$. \end{remark}

We will apply this construction not to the action of $G$ on $\CC$, but to the action of $G^*$ on the equivariant category $\CC_G$.

\begin{defn} Suppose one has a strict action of $G$ on $\CC$. Given an object $B \in \CC$, an \emph{equivariant structure} on $B$ is a \emph{compatible} system of isomorphisms $\phi_g \co
B \xrightarrow{\sim} gB$ for each $g \in G$. Here, the system $\phi$ is compatible if $g(\phi_h) \circ \phi_g = \phi_{gh} \co B \to ghB$ for all $g,h \in G$.
\[
\begin{tikzcd}
B \arrow{r}{\phi_g} \arrow[swap]{dr}{\phi_{gh}} & gB \arrow{d}{g(\phi_h)} \\ & ghB
\end{tikzcd}
\]
A \emph{morphism} of
equivariant structures from $(B,\phi)$ to $(B',\phi')$ is a morphism $f \in \Hom_{\CC}(B,B')$ such that $\phi_g' \circ f = g(f) \circ \phi_g \co B \to g(B')$ for all $g \in G$.
\[
\begin{tikzcd}
B \arrow{r}{\phi_g} \arrow[swap]{d}{f} & gB \arrow{d}{g(f)} \\
B' \arrow{r}{\phi'_g} & gB'
\end{tikzcd}
\]
The collection of equivariant structures $(B,\phi)$, and the morphisms between them, form the \emph{equivariant category} $\CC_G$. \end{defn}

Let $\CC$ be an additive (graded) $\CM$-linear category. Then $\CC_G$ is also an additive (graded) $\CM$-linear category. Let $G^* = \Hom(G,\CM^*)$ denote the Pontrjagin dual of $G$,
which is an abelian group.

\begin{defn} We define a natural strict action of $G^*$ on $\CC_G$. For any character $\xi \in G^*$ and any compatible system of isomorphisms $\phi = (\phi_g)_{g \in G}$ for an object $B
\in \CC$, there is a new compatible system of isomorphisms $\xi \cdot \phi$ given by the rescaling $(\xi \cdot \phi)_g = \xi(g) \phi_g$. Define $\xi(B,\phi) = (B, \xi \cdot \phi)$. If $f
\in \Hom(B,B')$ gives rise to a map $(B,\phi) \to (B',\phi')$, then $f$ also gives rise to a map $\xi(B,\phi) \to \xi(B',\phi')$. Thus $\xi \co \CC_G \to \CC_G$ is a functor. Moreover,
for $\xi, \eta \in G^*$, the composition of functors $\xi \circ \eta$ is actually equal (not just isomorphic) to the functor $\xi \eta$. Thus we can let $a_{\xi,\eta}$ be the identity
natural transformation. \end{defn}

When we apply Definition \ref{defn:wtdgrth} for $K = G^*$ and $\DC = \CC_G$, the character $\eta$ is actually an element of the Pontrjagin double dual $G^{**}$, which is canonically
isomorphic to the abelianization $G/[G,G]$.

\begin{defn} Given a strict action of $G$ on $\CC$ and an element $g \in G$, the \emph{$g$-weighted Grothendieck group} $[\CC]_g$ of $\CC$ will be the $\overline{g}$-weighted
Grothendieck group of $\CC_G$, where $\overline{g}$ is the class of $g$ in $G/[G,G]$. \end{defn}

\begin{remark} One can generalize these constructions slightly, either by changing the base ring of the category, or by changing the base ring of the decategorification. If $\CC$ is
$\Bbbk$-linear for some complete commutative local ring $\Bbbk$, one has an action of the abelian group $\Hom(G,\Bbbk^*)$ on $\CC_G$. Similarly, if $K$ acts on $\DC$ and $\eta \in
K^*$, then one can define the $\eta$-weighted Grothendieck group to be linear over any subring of $\CM$ containing the image of $\eta$.

The most common such situation is when $G = \ZM/2\ZM$, as all characters are defined over $\ZM$. There is an action of $G^*$ on $\CC_G$ even when $\CC$ is only
$\ZM$-linear, and one can consider a $\ZM$-form of the weighted Grothendieck group. \end{remark}

Suppose that $\CC$ is a monoidal category, and that the action of $G$ is by monoidal functors. Then $\CC_G$ is also a monoidal category, with $(B,\phi) \ot (B',\phi') = (B \ot
B', \phi \ot \phi')$. Moreover, for $\eta, \xi \in G^{*}$ one has $\eta(B,\phi) \ot \xi(B',\phi') = (\eta \xi)(B \ot B',\phi \ot \phi')$, so that the action of $G^*$ on $\CC_G$
respects the monoidal structure in the appropriate way. Therefore, for any $g \in G$, the $g$-weighted Grothendieck group of $\CC$ will inherit the structure of a ring.

Similarly, suppose that $\CC$ is equipped with a contravariant duality functor $\DM$, satisfying $\DM^2 \cong \IM$, and that the action of $G$ commutes with $\DM$ up to isomorphism. Then
$\CC_G$ can also be equipped with a duality functor, via $\DM (B,\phi) = (\DM B, \DM \phi)$. The structure map is given by $(\DM \phi)_g = \DM(g \phi_{g^{-1}}) \co \DM B \to \DM gB$. The
action of $\DM$ is conjugate-linear with respect to the action of $G^*$, in the sense that $\DM(\xi (B,\phi)) = \xi^{-1} \DM(B,\phi)$. In particular, the action of $\DM$ descends to the
$g$-weighted Grothendieck group to be $\CM$-conjugate-linear. When $\CC$ is graded and $\DM(M(1)) \cong \DM(M)(-1)$, the action of $\DM$ on the Grothendieck group is also $v$-antilinear.

Let $\For \co \CC_G \to \CC$ denote the forgetful functor which only remembers the underlying object/morphism. It is faithful. $\For$ induces a map of Grothendieck groups $[\CC_G] \to
[\CC]$, but it need not descend to $[\CC]_g \to [\CC]$. However, $\For$ clearly induces a map $[\CC]_e \to [\CC]$ for the identity element $e \in G$ (or for any
element $g \in [G,G]$).

\subsection{Equivariant mixed categories and indecomposable equivariant objects}
\label{subsec-indecomposable}

For a graded category $\CC$, linear over $\CM$, we write $\Hom^k(X,Y) = \Hom^0(X,Y(k))$ to denote the space of maps of degree $k$, and we write $\Hom(X,Y) = \oplus_{k \in \ZM}
\Hom^k(X,Y)$ for the graded vector space of maps of all degrees. Assume that $\Hom^k(X,Y)$ is finite dimensional for all $k$ and all $X,Y$. The graded dimension $\gdim \Hom(X,Y)$ will be
a polynomial $P \in \NM[v,v^{-1}]$, satisfying $P^k = \dim \Hom^k(X,Y)$.

\begin{defn} \label{defn:mixed} Consider an additive graded Karoubian $\Bbbk$-linear category $\CC$ (for some base field $\Bbbk$), equipped with a duality functor $\DM$ satisfying
$\DM(M(1)) = \DM(M)(-1)$. Let $\XM$ index the set of indecomposable objects, up to isomorphism and grading shift. Then $\CC$ is called \emph{mixed} if for each $x \in \XM$ there is a
choice of representative (i.e. choice of grading shift) $M_x \in \CC$ for which: \begin{itemize} \item $\DM(M_x) = M_{x^*}$ for some $x^* \in \XM$; \item $\gdim \Hom(M_x,M_y) = \de_{xy}
+ v\NM[v]$. \end{itemize} In particular, no object is isomorphic to a grading shift of itself, and $\CC$ has the graded Krull-Schmidt property. We say the mixed category is
\emph{self-dual} if $x^* = x$ for all $x \in \XM$. \end{defn}

This definition (or a similar one) is found, for instance, in \cite{WebsCan15}, which also contains many examples and a great deal of discussion.

Suppose that a finite group $G$ acts on a self-dual mixed category $\CC$, commuting with the grading shift and the duality functor. Clearly $G$ will preserve the set of self-dual
indecomposable objects in $\CC$. Thus $G$ will act on the set $\XM$, by $g M_x \cong M_{gx}$. Let $$M_{Gx} \define \oplus_{y \in G \cdot x} M_y,$$ a direct sum of the indecomposables in
the orbit of $x$. By the Krull-Schmidt property, for any $Y \in \CC$ one has \begin{equation} Y \cong gY \text{ for all } g \in G \qquad \iff \qquad Y \cong \bigoplus M_{Gx}(n) \text{ for
various } x \in \XM, n \in \ZM. \label{invariantobject} \end{equation} Only an object $Y$ satisfying \eqref{invariantobject} can be equipped with a $G$-equivariant structure.

Let us consider the possible equivariant structures which can be placed on $M_{Gx}^{\oplus n}$.

\begin{prop} Fixing an element $x$ in the orbit $Gx$, and fixing an arbitrary compatible system on $M_{Gx}$, one has an equivalence of categories between compatible systems on
$M_{Gx}^{\oplus n}$ for various $n$, and representations of the stabilizer $G_x$. Thus the indecomposable objects in $\CC_G$ are classified (after fixing some data), up to isomorphism and
grading shift, by an orbit $Gx \subset \XM$ and an irreducible representation of the stabilizer $G_x$ of $x$. \end{prop}

First we look at some examples.

\begin{example} Suppose that $G \cong \ZM / 4\ZM$, generated by the element $\si$. For an object $M \in \CC$, a compatible system of isomorphisms $\phi$ on $M$ would be determined by an
isomorphism $\phi_\si \co M \to \si M$. The compatibility of the system amounts to the fact that \begin{equation} \label{mod4compatible} \si^3(\phi_\si) \circ \si^2(\phi_\si) \circ \si(\phi_\si)
\circ \phi_\si = \phi_{\si^4} = 1_M \in \End^0(M). \end{equation}

An orbit of $G$ on $\XM$ has size 1, 2, or 4. Consider first an orbit $G \cdot x$ of size 1, yielding an indecomposable object $M = M_x$ such that $M \cong \si M$. Fix an arbitrary
compatible system, determined by an isomorphism $\al \co M \to \si M$. This yields an identification of the one-dimensional spaces $\Hom^0(M,\si M)$ and $\End^0(M)$, sending $\phi \in
\Hom^0(M,\si M)$ to $\psi = \al^{-1} \phi \in \End^0(M)$. Moreover, any element $\psi \in \End^0(M)$ is a scalar multiple of $1_M$, so that $\si(\psi)$ is that same scalar multiple of
$1_{\si M}$. From this we deduce that $\al \psi = \si(\psi) \al$ and $\al \psi^{-1} = \si(\psi^{-1}) \al$. Using this, it is not hard to show that the relation \eqref{mod4compatible} on a map
$\phi_\si \in \Hom^0(M,\si M)$ is equivalent to the relation $$\psi^4 = 1_M.$$ Therefore, there are four equivariant structures on $M$, given by $\psi = \ze_4^k 1_M$ for $k = 0,1,2,3$. It
is easy to see that these four structures are non-isomorphic in $\CC_G$. The chosen compatible system $\al$ corresponds to $k=0$.

Now consider an orbit of size 4, yielding a sum $M = M_x \oplus M_{\si x} \oplus M_{\si^2 x} \oplus M_{\si^3 x}$ of four distinct indecomposables, such that $M \cong \si M$. Fix an arbitrary
compatible system, giving an identification of $\si(M_x)$ with $M_{\si x}$. A map $\phi_\si \co M \to \si M$ is determined by a quadruple of scalars $a,b,c,d \in \CM$, corresponding to the
maps $a 1_{M_x} \in \End^0(M_x) = \Hom^0(\si(M_x),M_{\si x})$, $b 1_{M_{\si x}} \in \End^0(M_{\si x})$, etcetera. The compatibility condition states that $abcd=1$. There is one equivariant
structure on $M$ for each such quadruple. However, which quadruple $a,b,c,d$ is chosen is irrelevant up to change of basis, or in other words, $(M,a,b,c,d) \cong (M,a',b',c',d')$ in
$\CC_G$, for any two quadruples. Therefore, there is only one isomorphism class of equivariant structure on $M$.

For an orbit of size 2, the argument is a mix of the two cases above. After choosing an isomorphism $\al \co M \to \si^2 M$, the map $\phi_\si$ is determined by a pair of scalars $a, b \in
\CM$; compatibility states that $(ab)^2 = 1$; changing basis can alter the individual values of $a$ and $b$, but will not change the product $ab$. Therefore, there are two isomorphism
classes of equivariant structure on $M$, given by $ab = \ze_2^k$ for $k = 0,1$. \end{example}

Note that these arguments rely heavily on the fact that $\End^0(M_x) \cong \CM$ for any indecomposable object $M_x$. Without this assumption, equivariant objects can be more complex. This
is one of the crucial features of mixed categories.

Now we sketch the proof of the proposition.

\begin{proof} Fix $x \in \XM$ and consider an equivariant object $(M_{Gx} \sqot V,\phi)$, for some multiplicity space $V$. By fixing an arbitrary $G_x$-equivariant structure on $M_x$,
it is easy to see that $\phi$ induces a representation of $G_x$ on $V$, in similar fashion to the examples above (where $\al$ determined the $G_x$-equivariant structure). It is an
exercise to see that there is a morphism between equivariant objects precisely when there is a morphism of $G_x$ representations, and that these morphisms compose appropriately.
\end{proof}

Note that this classification depended on an arbitrary choice of $x$ and a $G_x$-equivariant structure on $M_x$. In the example above for an orbit of size 1, the
indecomposable attached to the trivial representation of $G$ depended upon a choice of isomorphism $\al$. Once one chooses an equivariant structure to correspond to the trivial
representation, the equivariant structures attached to other representations are determined. After making this choice, we denote the indecomposable objects in $\CC_G$ by $(M_{Gx} \sqot
V)$, for a representation $V$ of $G_x$. We let $(M_{Gx} \sqot \CM)$ or $(M_{Gx},+1)$ denote the irreducible corresponding to the trivial representation.

Because $\For$ is faithful, it is easy to see that morphism spaces between various $(M_{Gx} \sqot V)$ are concentrated in non-negative degree, with all degree zero morphisms being
isomorphisms. Moreover, $\DM(M_{Gx} \sqot V) \cong (M_{Gx} \sqot V^*)$. Therefore $\CC_G$ is a mixed category (though not necessarily self-dual), and in particular it has the
Krull-Schmidt property.

In future sections, we will write $[M_{Gx},V]$ for the image of $(M_{Gx} \sqot V)$ in a Grothendieck group or weighted Grothendieck group.

\subsection{Nonabelian examples}
\label{subsec-nonabelian}

Let $G$ be an arbitrary group acting on a self-dual mixed category $\CC$, and let $g \in G$. The $g$-weighted Grothendieck group is (by the above definition) the quotient of $[\CC_G]$ by
the relation \begin{equation} \label{eq:gweighted} [M_{Gx},\xi \ot V] = \xi(g) [M_{Gx},V] \end{equation} for any $\xi \in G^*$ and $V \in \Rep(G_x)$. We provide some non-abelian examples.

\begin{example} Let $G$ be nonabelian, let $H = [G,G]$ be the commutator subgroup, and let $K$ be the abelianization $G/H$. Suppose that $\XM$ is identified with $K$ as a $G$-set. Then
$\XM/G$ has one orbit, with stablizer $H$, and any character of $G$ will fix any representation of $H$. There are two possibilities for the weighted Grothendieck group. If $g \in H$ then
any relation of the form \eqref{eq:gweighted} will be trivial, as $\xi(g) = 1$ and $\xi \ot V = V$ for any $\xi \in G^*$ and $V \in \Rep(H)$. Therefore, $[\CC]_g = [\CC_G] = [\Rep(H)]$.
On the other hand, if $g \notin H$ then there exists some character $\xi \in G^*$ such that $\xi(g) \ne 1$. Then the relation $[V] = \xi(g) [V]$ will force $[V]=0$ for all $V \in
\Rep(H)$. Therefore $[\CC]_g = 0$. \end{example}

\begin{example} Let $G = S_n$, acting by the standard action on a set $\XM$ of size $n$. Therefore, a given stabilizer is isomorphic to $S_{n-1}$. The only non-trivial character of
$S_n$ is the sign representation, which descends to the sign representation of any stabilizer. The irreducible representations $V_\la$ of $S_{n-1}$ are parametrized by partitions $\la$
of $n-1$, and the action of the sign representation switches $\la$ with its transpose partition $\la^t$. If $g$ is even then $[\CC]_g$ will have a basis given by partitions modulo
transpose, with $[V_\la] = [V_{\la^t}]$. If $g$ is odd, then $[\CC]_g$ has a basis given by non-self-transpose partitions modulo transpose, with $[V_\la] = - [V_{\la^t}]$. \end{example}

\subsection{Decomposition in equivariant mixed categories}
\label{subsec-decomposing}

Let $M$ be an arbitrary object in a self-dual mixed category $\CC$. Within $\CC$ there is a direct sum decomposition $$M \cong \bigoplus_{x \in \XM, k \in \ZM} M_x(k)^{\oplus d_{x,k}}$$
for some multiplicities $d_{x,k}$. When $M$ is self-dual one has $d_{x,k} = d_{x,-k}$. To compute the numbers $d_{x,k}$, one uses the \emph{local intersection pairing} or \emph{LIP},
which is the pairing given by composition \[\Hom^{-k}(M_x,M) \times \Hom^k(M,M_x) \to \End^0(M_x) = \CM. \] Then $d_{x,k}$ will be the rank of this pairing. We let $V^k(M,M_x)$ denote the
quotient of $\Hom^k(M,M_x)$ by the kernel of this pairing, a vector space of dimension $d_{x,k}$.

\begin{remark} It is not difficult to observe that the kernel of the local intersection pairing is contained in the Jacobson radical of $\CC$. In fact, the union of all such kernels will
generate the Jacobson radical, and modulo the Jacobson radical all morphisms will split. Any family of endomorphisms which become orthogonal idempotents modulo the Jacobson radical can
be lifted to a family of orthogonal idempotents, so that questions of decomposition can be answered modulo the Jacobson radical, where $V^k(M,M_x)$ is a genuine Hom space. We will not pursue this style of algebra in this paper.\end{remark}

\begin{remark} \label{rmk:pospolys} Suppose that $\CC$ is monoidal and that the monoidal identity $\IM$ is a self-dual indecomposable object. Its endomorphism ring is a
non-negatively-graded commutative ring $\End(\IM)=R$, with only scalars in degree $0$. This ring will act on the right and the left of any morphism space $\Hom(M,N)$, and all composition
maps are $R$-bimodule maps. The action of the positive part $R_+$ will send any morphism into the kernel of the LIP. The reason is that the image of (either the right or left action of)
$R_+$ in $\End(M_x)$ does not intersect $\End^0(M_x)$. \end{remark}

Now let $(M,\phi)$ be an arbitrary object in $\CC_G$. Decomposing $(M,\phi)$ in $\CC_G$ is more work than decomposing $M$ in $\CC$: a decomposition of representations, rather than a
dimension count. For any $x \in \XM$, let $\psi$ be the chosen equivariant structure map for $(M_{Gx},+1)$. The vector space $V^k(M,M_x)$ admits an action of the stabilizer $G_x$ as
follows: for $g \in G_x$ and $f \in V^k(M,M_x)$, let \begin{equation} \label{Gxaction} g \cdot f = g(\psi_{g^{-1}}) \circ g(f) \circ \phi_g \co M \to gM \to g(M_x) \to M_x. \end{equation}

\begin{claim} \label{claim:multiplicity} For a representation $V$ of $G_x$, the multiplicity of $(M_{Gx} \sqot V)$ in $(M,\phi)$ is equal to the multiplicity of $V$ in $V^k(M,M_x)$. We
call this multiplicity $d_{x,V,k}$. \end{claim}

\begin{proof} This is a fairly straightforward argument, which we leave to the reader. The key point is that a morphism in a $V$-component of $V^k(M,M_x)$ will help produce an equivariant
morphism from $(M,\phi)$ to $(M_{Gx} \sqot V)$. \end{proof}

Our reason for computing decompositions in $\CC_G$ will mainly be to compute the ring structure on $[\CC]_g$.

\subsection{The abelian case}
\label{subsec-abelian}

Abelian groups have several advantages which contribute to the ease in understanding of their weighted Grothendieck groups. These advantages are \begin{itemize} \item For any subgroup $H
\subset G$ (such as a stabilizer subgroup), the action of $G^*$ on $H^*$ is transitive. \item All representations are characters, so that the action of $H^*$ on the objects in
$\textrm{Rep}(H)$ is transitive. \item For any subgroup $H$, $G/H$ is abelian. Therefore, for any $g \notin H$, there is a character $\xi \in G^*$ such that $\xi(g) \ne 1$ but $\xi(H)=1$.
\end{itemize}

When we write $x \in \XM/G$, we mean that we choose a representative for an orbit. Because $G$ is abelian, the stabilizer $G_x$ is independent of the chosen orbit representative $x$. For
a character $\xi \in G^*$ we write $(M_{Gx},\xi)$ for the corresponding indecomposable object, viewing $\xi$ as a character of $G_x$. We write $[M,\phi]$ for the class of an object
$(M,\phi)$ in some weighted Grothendieck group of $\CC$.

\begin{prop} \label{prop:basis} Let $G$ be an abelian group and let $g \in G$. The $g$-weighted Grothendieck group $[\CC]_g$ has a basis $\{ [M_{Gx},+1] \}$ in bijection with the
orbits $Gx$ which $g$ fixes pointwise, that is, for which $g \in G_x$. \end{prop}

\begin{proof} Certainly the ordinary Grothendieck group $[\CC_G]$ has a basis given by \[ \{[M_{Gx}, \xi]\}_{x \in \XM/G, \xi \in G_x^*}, \] which splits as a vector space into subspaces
spanned by $[M_{Gx},\xi]$ for each fixed $x$. The relations imposed by the $g$-weighting preserve this decomposition. One has $[M_{Gx},\xi] = \xi(g) [M_{Gx},+1]$, so that certainly
$[M_{Gx},+1]$ spans its corresponding subspace, which is therefore at most one-dimensional. Additional relations can only be imposed by characters in $(G/G_x)^*$.

If $g \notin G_x$ then there exists some character $\xi \in G^*$ with $\xi(g) \ne 1$ and $\xi(G_x)=1$. Therefore $\xi$ preserves $(M_{Gx},+1)$, so $[M_{Gx},+1] = \xi(g) [M_{Gx},+1]$,
implying that $[M_{Gx},+1]=0$. On the other hand, if $g \in G_x$ then any character which preserves $(M_{Gx},+1)$ will also send $g$ to $1$, so that no non-trivial relations exist on
$[M_{Gx},+1]$. \end{proof}

\begin{cor} \label{cor:forgetful} The forgetful functor $\For \co \CC_G \to \CC$ induces an isomorphism $[\CC]_e \cong [\CC]^G$. \end{cor}

\begin{proof} By the previous proposition, $[\CC]_e$ has a basis given by $\{ [M_{Gx},1] \}$ ranging over all orbits $x \in \XM/G$. The basis element $[M_{Gx},+1]$ is sent to $\sum_{y \in
Gx} [M_y]$, and together these images form a basis for $[\CC]^G$. \end{proof}

The following corollary will help to quickly compute the ring structure on $[\CC]_g$.

\begin{cor} \label{cor:trace} For any $(M,\phi) \in \CC_G$, one has the following equality in $[\CC]_{g}$: \begin{equation} \label{eq:trace} [M,\phi] = \sum_{k \in \ZM} \sum_{x \in \XM/G \textrm{ such that } gx = x} \Tr_{g} V^k(M,M_x) \cdot v^k [M_{Gx},+1]. \end{equation} \end{cor}

\begin{proof} We have already seen in Claim \ref{claim:multiplicity} how $(M,\phi)$ splits into indecomposables in $\CC_G$. Let $V$ be a representation of $G$ for which the character
$\xi$ appears with multiplicity $d_\xi$. Then $\sum d_\xi \xi(g) = \Tr_g V$. The corollary follows easily. \end{proof}

\subsection{Reductions in the abelian case}
\label{subsec-reductions}

Consider the situation where $G$ is an abelian group acting on $\CC$ trivially. The category $\CC_G$ is much larger than $\CC$, having an object $(M,\xi)$ for each object $M \in \CC$
and character $\xi \in G^*$. However, Proposition \ref{prop:basis} implies that, for any $g \in G$, the $g$-weighted Grothendieck group $[\CC]_g$ is isomorphic to the usual Grothendieck
group $[\CC]$ as a vector space; Corollary \ref{cor:trace} implies that they are isomorphic as rings when $\CC$ is monoidal. For purposes of the weighted Grothendieck group, we may as
well assume that $G$ is trivial. In this section, we consider several similar reductions.

Let $G$ be an abelian group acting on a self-dual equivariant mixed category $\CC$. Let $\KC \subset \CC$ be a monoidal Karoubian subcategory, closed under the action of $G$. Then $\KC_G
\subset \CC_G$ is fully faithful, and this inclusion induces an injection of weighted Grothendieck groups which preserves the ring structure.

Let $K \subset G$, and suppose that it acts trivially (on both morphisms and objects) on a subcategory $\KC$ as above. Let $G' = G/K$. The equivariant categories $\KC_G$ and $\KC_{G'}$
are certainly not equivalent. However, the results above imply that they have the same $g$-weighted Grothendieck groups, for an element $g \in G$ viewed as an element of either $G$ or
$G'$.

Combining these previous two observations allows us to simplify slightly when computing the ring structure on $[\CC]_g$. For instance, if we want to compute the product $[M_{Gx},+1]
[M_{Gy},+1]$, we may work within the smallest $G$-closed monoidal Karoubian subcategory $\KC$ which contains both $M_{Gx}$ and $M_{Gy}$. Then we may assume that $G$ acts faithfully on
$\KC$, because we may replace $G$ with the quotient by the kernel of the action.

\begin{remark} More general statements can be made, because the condition that $K$ acts trivially on objects is not a categorical one, and can be weakened. However, this reduction will
help us in our discussion of folding below. In this case, one has a mixed category $\DC$ attached to any Coxeter graph, with a monoidally-closed subcategory attached to any subgraph.
The group $G$ will act by automorphisms of a Coxeter graph, and for any union of orbits in the Coxeter graph there will be a subcategory like $\KC$ above. Any subgroup $K$ acting
trivially on this union of orbits will act trivially on the corresponding subcategory. Thus when computing in the weighted Grothendieck group, we can replace $G$ with its image in the
permutation group of these orbits.

We must clarify this statement slightly. This comment may not make sense until one reads the chapter on Soergel diagrammatics. If a subgroup $K$ acts trivially on a certain union of
orbits in a Coxeter graph, but non-trivially somewhere, then $K$ does not actually act trivially on morphism spaces in $\KC \subset \DC$. This is because $\IM \in \KC$ and $\End(\IM)$
is a graded polynomial ring $R$, whose positive part depends on the entire Coxeter graph and not just the union of orbits. However, as noted in Remark \ref{rmk:pospolys}, the image of
the action of $R_+$ on any Hom space is in the kernel of the LIP, and therefore will not contribute to any decomposition computations. \end{remark}

Here is another simplication we can use. Suppose that $G$ is abelian, and that some element $g \in G$ acts transitively on each orbit in $\XM$. Let $H \subset G$ be the cyclic subgroup
generated by $g$. We now compare the $g$-weighted Grothendieck group of $\CC$, when acted on by $G$ or by $H$. The orbit spaces $\XM/G$ and $\XM/H$ are the same, and only singleton orbits
will contribute to the $g$-weighted Grothendieck group in either case. The trace of $g$ on $V^k(M,M_x)$ also is unaffected by the choice $G$ versus $H$. Therefore, $[\CC]_g$ is
independent of whether $\CC$ is acted on by $G$ or by $H$. So if weighted by such an element, we may as well assume that $G$ is cyclic. However, once again, the actual equivariant
categories $\CC_G$ and $\CC_H$ are quite different.

\begin{remark} The combined assumptions that $G$ is abelian, $g \in G$ acts transitively on each orbit, and $G$ acts faithfully, are not sufficient to imply that $G = H$. One could
have $G = \ZM/3\ZM \times \ZM/3\ZM$ and $\si = (1,1)$, acting on a disjoint union of two copies of $\ZM/3\ZM$. \end{remark}

As seen in the non-abelian examples above, one can not make any of these reductions when $G$ is not abelian. For example, the standard action of $S_n$ has a cyclic subgroup which acts
transitively, but the weighted Grothendieck groups of $S_n$ and this cyclic subgroup are very different.

Suppose $G$ is non-abelian and $K \subset G$ acts on $\DC$ trivially. Consider a tensor product $M$ of indecomposables $(M_{Gx} \sqot V)$ in $\DC_G$. Suppose that $K$ acts trivially on
each representation $V$ of $G_x$. Then $K$ will act trivially on each space $V^k(M,M_{Gx})$, and on every representation appearing in the decomposition of $M$. Therefore, the full
subcategory $\DC_G^K \subset \DC_G$ whose indecomposable summands are $(M_{Gx} \sqot V)$, where $V$ is a representation of $G_x$ on which $K$ acts trivially, is actually a monoidal
subcategory. Thus, we may assume that $G$ acts faithfully for some computations, such as computing decompositions and the monoidal structure within $\DC_G^K$. Unfortunately, this will
not help when it comes time to compute the ring structure on the weighted Grothendieck group, as the weighted relation will not ``preserve" $\DC_G^K$.

\section{Soergel bimodules}
\label{sec-sbim}

In 1990, Soergel \cite{Soer90} defined a graded monoidal category of bimodules over a coinvariant ring of a Weyl group, known now as Soergel bimodules. This monoidal category has
Grothendieck ring isomorphic to the Hecke algebra. Since then, Soergel has modified his construction to yield bimodules over a polynomial ring associated to any Coxeter group, which
categorify the Hecke algebra $\HC(W,S,\ell)$ of that Coxeter group. See \cite{Soer07} for more details.

However, in \cite{EWGr4sb}, the authors have given another description of the same monoidal category, in a fashion which does not reference bimodules at all. This description is by
generators and relations, and uses the language of planar diagrammatics. We refer the reader to \cite{EWGr4sb} for additional background information and references. Definition
\ref{defn:DC} is almost copied verbatim from \cite{EWGr4sb}. For experts who do not need the review, it is important to note Claim \ref{claim distant}, and to recall some relevant features of the category outlined in \S\ref{subsec-catfnthms}.

In this chapter we fix a Coxeter group $(W,S)$. (This will be the Coxeter group $(W',S')$ into which another Coxeter group will be embedded, but there is no need for apostrophes at the
moment.)

\subsection{Realizations}
\label{subsec-cartan}

The input to Soergel's construction (as modified in \cite{EWGr4sb}) is a \emph{realization} of $(W,S)$, which is roughly a generalization of the reflection representation of $W$. A
realization is the data of a representation $\hg$ of $W$ over a base ring $\Bbbk$, together with a choice of simple roots $\al_s \in \hg^*$ and simple coroots $\al_s^\vee \in \hg$,
satisfying certain nondegeneracy conditions. One requires that the action of a simple reflection on $\hg$ and $\hg^*$ is given by ``reflection,'' so that for $v \in \hg^*$ one has $s(v)
= v - \al_s^\vee(v) \al_s$.

Let $R = \OC(\hg) = \Sym[\hg^*]$ denote the polynomial ring over $\Bbbk$ with linear terms equal to $\hg^*$. The ring $R$ inherits an action of $W$. We equip $R$ with a grading such
that linear terms have degree $2$.

\subsection{Diagrammatics for Soergel bimodules}
\label{subsec-diagrammatics}

\begin{defn} An \emph{$S$-graph} is a certain kind of colored planar graph with boundary, properly embedded in the strip $\RM \times [0,1]$. In other words, all vertices of the
graph lie on the interior of the strip, and edges may terminate either at a vertex or on the boundary of the strip. Additionally, an edge may form a closed loop. The number of
vertices and the number of components are required to be finite. The edges of this graph are unoriented, and colored by elements $s \in S$. We call the place where an edge meets the
boundary a \emph{boundary point}; boundary points are not vertices. The boundary points on $\RM \times \{0\}$ (resp. $\{1\}$) give a finite sequence of colored points, known as the
\emph{bottom boundary} (resp. \emph{top boundary}) of the graph.

The only vertices allowed in an $S$-graph are of 3 types (see Figure \ref{thegenerators}):

\begin{itemize} \item Univalent vertices (dots). These have degree $+1$. \item Trivalent vertices, where all three adjoining edges have the same color. These have degree $-1$. \item
$2m$-valent vertices, where the adjoining edges alternate in color between two elements $s \ne t \in S$, and $m_{st}=m<\infty$. These have degree $0$. \end{itemize}

We also allow decorations called \emph{polynomials} to float in the regions of the graph. These will be labeled by homogeneous polynomials $f \in R$. The \emph{degree} of an $S$-graph
is the sum of the degree of each vertex, and the degree of each polynomial.

We say that two Soergel graphs are \emph{isotopic} if there is an isotopy of Soergel graph embeddings between them, relative to the boundary. Such an isotopy can alter the precise
location of the boundary points on $\RM \times\{0,1\}$, but can not change the bottom or the top boundary sequence. \end{defn}

\begin{figure} \label{thegenerators} \caption{The vertices in an $S$-graph} $\ig{1}{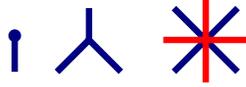}$ \end{figure}

\begin{defn} \label{defn:DC} Let $\DC$ (or if there is ambiguity, $\DC(W,S)$) denote the $\Bbbk$-linear monoidal category defined as follows. Objects are sequences $\un{w}$, sometimes
denoted $B_{\un{w}}$, with monoidal structure given by concatenation. The space $\Hom_{\DC}(\un{w},\un{y})$ is the free $\Bbbk$-module generated by $S$-graphs with bottom boundary
$\un{w}$ and top boundary $\un{y}$, modulo the relations below. Hom spaces will be graded by the degree of the graphs, and all the relations below are homogeneous.

The first (unwritten) relation is that two isotopic $S$-graphs are equal.

\emph{The polynomial relations:}
\begin{gather}
\label{alphais}
\begin{array}{c}
\tikz[scale=0.9]{
\draw[dashed,gray] (0,0) circle (1cm);
\draw[color=red] (0,0.5) -- (0,-0.5);
\node[draw,fill,circle,inner sep=0mm,minimum size=1mm,color=red] at (0,0.5) {};
\node[draw,fill,circle,inner sep=0mm,minimum size=1mm,color=red] at (0,-0.5) {};
} \end{array}
=
\begin{array}{c}
\tikz[scale=0.9]{
\draw[dashed,gray] (0,0) circle (1cm);
\node at (0,0) {$\alpha_s$};
} \end{array}, \\
\label{dotforcegeneral} 
\begin{array}{c}
\tikz[scale=0.9]{
\draw[dashed,gray] (0,0) circle (1cm);
\draw[color=red] (0,-1) to (0,1);
\node at (-0.5,0) {$f$};
} \end{array}
=
\begin{array}{c}
\tikz[scale=0.9]{
\draw[dashed,gray] (0,0) circle (1cm);
\draw[color=red] (0,-1) to (0,1);
\node at (0.5,0) {${\color{red}s}(f)$};
} \end{array}
+
\begin{array}{c}
\tikz[scale=0.9]{
\draw[dashed,gray] (0,0) circle (1cm);
\draw[color=red] (0,-1) to (0,-0.5);
\draw[color=red] (0,1) to (0,0.5);
\node[color=red,draw,fill,circle,inner sep=0mm,minimum size=1mm] at (0,-0.5) {};
\node[color=red,draw,fill,circle,inner sep=0mm,minimum size=1mm] at
(0,0.5) {};
\node at (0,0) {$\partial_{\color{red}s} f$};
} \end{array}.
\end{gather}

\emph{The one color relations:}
\begin{equation} \label{assoc1} \ig{1}{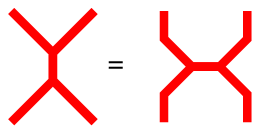} \end{equation}	
\begin{equation} \label{unit} \ig{1}{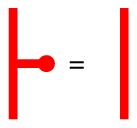} \end{equation}
\begin{equation} \label{needle} \ig{1}{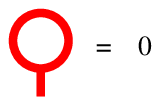} \end{equation}

\emph{The two color relations:} 
The color scheme depends slightly on the parity of $m=m_{st}<\infty$. We give one example of each relation for each
parity; the reader can guess the general form.

\begin{equation} \label{assoc2} \ig{1}{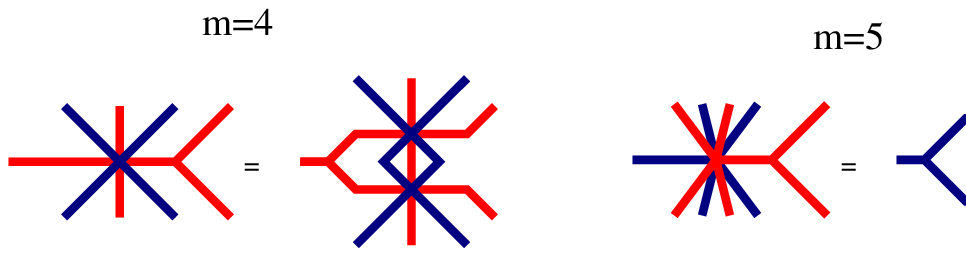} \end{equation}
\begin{equation} \label{dot2m} 	{
	\labellist
	\small\hair 2pt
	 \pinlabel {$JW_{m-1}$} [ ] at 101 16
	 \pinlabel {$JW_{m-1}$} [ ] at 242 16
	\endlabellist
	\centering
	\ig{1}{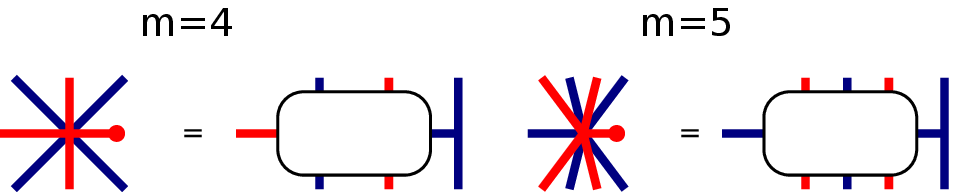}
	} \end{equation}

In equation \eqref{dot2m} above, the \emph{Jones-Wenzl morphism} $JW_{m-1}$ is a $\Bbbk$-linear combination of graphs constructed only out of dots and trivalent vertices. For details,
see \cite{EWGr4sb}.

\emph{The three color relations:} It will be clear from the graphs which colors represent which indices.

For a triplet of colors forming a sub-Coxeter system of type $A_1 \times I_2(m)$ for $m<\infty$, we have

\begin{eqnarray} (A_1 \times I_2(m)) && \ig{1}{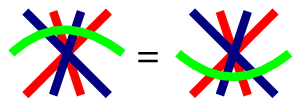} \label{A1I2m} \end{eqnarray}

A specific example, when $m=2$, is the case $A_1 \times A_1 \times A_1$:

\begin{eqnarray} (A_1 \times A_1 \times A_1) && \ig{1}{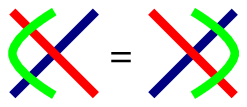} \label{A1A1A1} \end{eqnarray}

The last three relations are for types $A_3$, $B_3$, and $H_3$ respectively, and are known as the \emph{Zamolodzhikov relations}.

\begin{eqnarray} (A_3) && \ig{1}{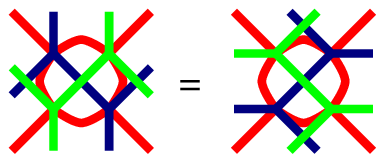} \label{A3} \end{eqnarray}

\begin{eqnarray} (B_3) && \ig{1}{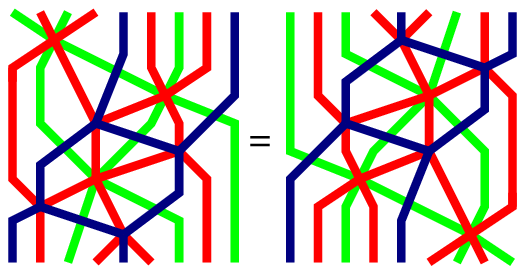} \label{B3} \end{eqnarray}

\begin{eqnarray} (H_3) && 
\input{H3left.tex}
- 
\input{H3right.tex}
= 
\text{lower terms}
\label{H3} \end{eqnarray}

This concludes the definition of $\DC$. \end{defn}

For any $w \in W$, the space of all (linear combinations of) morphisms which factor through objects $B_y(n)$ for $y < w$ is an ideal in $\DC$, called $\DC_{< w}$. When the element $w$ is
understood, elements of this ideal are called \emph{lower terms}. For instance, in \eqref{H3} above, $w=w_0$ is meant to be the longest element of $H_3$. Unfortunately, the relation
\eqref{H3} is only in general form, as the coefficients of the lower terms which appear are not known. Thankfully, the precise form of \eqref{H3} will not be relevant for this paper.

Given an element $w \in W$, one can construct the (oriented) graph $\Ga_w$ whose vertices are the reduced expressions of $w$. There is a (bidirectional) edge in $\Ga_w$ between two
expressions if they differ by a single application of a braid relation. There is a functor from this graph to $\DC$, which sends a reduced expression $\un{w}$ to the object
$B_{\un{w}}$, and sends an edge to the appropriate $2m$-valent vertex. To a path in $\Ga_w$ one associates a morphism in $\DC$ known as a \emph{rex move}. Two distinct rex moves with the
same source and target need not be equal as morphisms in $\DC$, but they are equal modulo lower terms. For example, all the three color relations can be thought of as equalities between
two different rex moves in $\Ga_{w_0}$ (or in the case of $H_3$, equality modulo lower terms).

Note that (graded) Hom spaces are enriched in graded $R$-bimodules, since one can place a polynomial in the leftmost or rightmost region.

\begin{remark} \label{remark distant} Let us observe some of the consequences of these relations for \emph{distant colors} $s,t \in S$ with $m_{st}=2$. In this case, the two
color relations are

\begin{equation} \label{assoc2 distant} \ig{1}{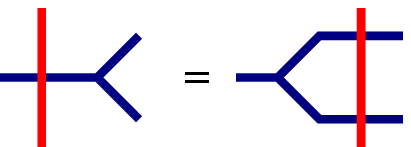} \end{equation}
\begin{equation} \label{dot2m distant} \ig{1}{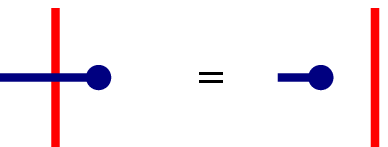} \end{equation}

A consequence of these two relations is a ``pulling apart" relation, which implies that the objects $B_s B_t$ and $B_t B_s$ are isomorphic in $\DC$.

\begin{equation} \label{R2} \ig{1}{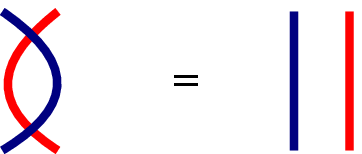} \end{equation}

Combining these relations with the three color relations of type $A_1 \times I_2(m)$, one can conclude the following. \end{remark}

\begin{claim} \label{claim distant} Consider $S$-graphs whose colors lie in $S_1 \cup S_2 \subset S$, and contain no polynomials. Assume that every color in $S_1$ is distant from
every color in $S_2$. We say that two such $S$-graphs with the same boundary are \emph{distant-isotopic} if their $S_1$-colored subgraphs are isotopic, and their $S_2$-colored
subgraphs are isotopic. Then any two distant-isotopic $S$-graphs represent equal morphisms in $\DC$. \end{claim}

\subsection{Categorification results}
\label{subsec-catfnthms}

Let $\Kar(\DC)$ denote the Karoubi envelope of $\DC$, a graded monoidal category. The following result, known as the Soergel Categorification Theorem, was originally proven by
Soergel. An easier proof was given by the authors in slightly more generality in \cite{EWGr4sb}.

\begin{thm} \label{SCT} The category $\Kar(\DC)$ is Krull-Schmidt. The indecomposable objects are $\{B_w(n)\}$ for $n \in \ZM$, parametrized up to shift by $w \in W$. Moreover, $B_w$
is characterized as the unique summand of $B_{\un{w}}$, for a reduced expression of $w$, which is not also a summand of $B_{\un{y}}$ for any shorter sequence $\un{y}$. The map $$b_s
\mapsto [B_s]$$ gives an isomorphism of $\Zvv$-algebras from $\HC(W,S,\ell)$ to $[\Kar(\DC)]$. The character map (defined below) gives an inverse isomorphism. Morphism spaces in
$\Kar(\DC)$ are free as left or right $R$-modules, and counting the graded rank of morphism spaces gives the standard pairing on $\HC$: $$([M],[N]) = \grank \Hom_{\DC}(M,N).$$
\end{thm}

Let us record a number of significant features of these indecomposable bimodules. \begin{itemize} \item The category $\DC$ also admits a duality functor, which we do not discuss. This
duality descends in the Grothendieck group to the bar involution on $\HC(W,S,\ell)$. The indecomposables $B_w$ are self-dual. \item Suppose one has two reduced expressions $\un{x}$ and
$\un{y}$ for $w \in W$. Then $B_{\un{x}}$ and $B_{\un{y}}$ each have a unique summand which does not factor through lower terms. However, these two summands are canonically isomorphic,
and this isomorphism is induced by any rex move from $\un{x}$ to $\un{y}$. \item In the Grothendieck group, one has $[B_{\un{w}}] = [B_w] + \sum_{y < w} m_{\un{w},y} [B_y]$, which is an
upper triangular decomposition. In particular, this implies that the Grothendieck group is generated as an algebra by $[B_s]$ for $s \in S$. \item Suppose that $I \subset S$ is a set of
mutually commuting reflections. Let $\un{x}$ be a sequence involving each element of $I$ once; it is a reduced expression for the longest element $w_I$ of the parabolic subgroup $W_I$.
Then $B_{w_I} \cong B_{\un{x}}$.\end{itemize}

The following theorem, originally known as the Soergel conjecture, was proven in generality by the authors in \cite{EWHodge}.

\begin{thm} Suppose that the realization is obtained by base change from a realization defined over $\RM$. Under the isomorphism $\HC(W,S,\ell) \to [\Kar(\DC)]$, one has $b_w \mapsto [B_w]$ for all $w \in W$. \end{thm}

One of the most significant implications of this theorem is the graded orthonormality of Hom spaces (c.f. section \ref{subsec-Heckeeq}). In other words, $\End(\oplus_{w \in W} B_w)$ is
concentrated in non-negative degree, and the only degree zero maps are (scalar multiples of) identity maps. Therefore, $\Kar(\DC)$ is a mixed category.

Given any object $M \in \Kar(\DC)$, let $P_{w,M} \in \NM[v,v^{-1}]$ denote the graded dimension of the morphism space $\Hom(M,B_w)$ modulo $\DC_{< w}$. For any choice of reduced
expression $\un{w}$, this morphism space is isomorphic to $\Hom(M,B_{\un{w}})$ modulo lower terms. The sum $\sum_{w \in W} P_{w,M} H_w \in \HC(W,S,\ell)$ is called the \emph{character} of
$M$. Character is obviously additive under direct sums, and thus yields a $\Zvv$-module map $[\Kar(\DC)] \to \HC(W,S,\ell)$. As Theorem \ref{SCT} states, this gives the inverse
isomorphism to the map $b_s \mapsto [B_s]$.

Let us fix a reduced expression $\un{w}$. For any sequence $\un{x}$, Libedinsky \cite{LibLL} gives a (non-canonical) combinatorial construction of morphisms in $\Hom(B_{\un{x}},
B_{\un{w}})$, called \emph{light leaves}, which descend to a (homogeneous) basis of $\Hom(B_{\un{x}},B_w)$ modulo $\DC_{< w}$. We give light leaves a diagrammatic interpretation in
\cite{EWGr4sb}. For an idempotent $e \in \End(B_{\un{x}})$, one can compute the character of the image of $e$ by computing the graded dimension of $\Hom(B_{\un{x}},B_{\un{w}}) e$ modulo
$\DC_{< w}$. This is not easy to do in general, but light leaves make the computation of individual cases fairly explicit.

Libedinsky's light leaves can be used to construct a cellular basis of $\DC$, known as the basis of \emph{double leaves}. For an reduced expression $\un{w}$, the identity of
$\End(B_{\un{w}})$ can be considered to span the highest cell of this endomorphism ring. The idempotent which projects from $B_{\un{w}}$ to $B_w$ is the unique idempotent which, expressed
in the double leaves basis, has a non-zero coefficient of the identity; moreover, this coefficient is 1.

\section{Folding}
\label{sec-folding}

Fix a $G$-split embedding $(W,S,L) \to (W',S',\ell)$. Let $\DC$ denote the category defined above for $(W',S')$. Then there is a strict action of $G$ on $\DC$ given by permuting colors.
In other words, $G$ acts on the generating objects by sending $B_s$ to $B_{g(s)}$; $G$ acts on $S$-graphs by permuting the colors of edges. We assume henceforth that the realization of
$(W',S')$ used to define $\DC$ is the base change of a realization defined over $\RM$.

As discussed in section \ref{subsec-catfnthms}, the Soergel conjecture implies that $\Kar(\DC)$ is mixed. This in turn allows us to apply the results of section
\ref{subsec-abelian}, and to state a great deal about the equivariant category $\DC_G$. For instance, Corollary \ref{cor:forgetful} already implies the following result.

\begin{prop} If $G$ is abelian, then there is an isomorphism of $\Zvv$-algebras from $[\DC]_e$ to $\HC(W',S',\ell)^G$. \end{prop}

Let us be precise in naming the indecomposable objects of $\DC_G$. Let $\un{w}$ be a reduced expression for $w \in W'$, and let us fix an inclusion $B_w \sumset B_{\un{w}}$ (this is
unique, up to scalar). For any $g \in G$, this also fixes an inclusion $B_{g(w)} \sumset B_{g(\un{w})}$. Suppose that $g(w)=w$, so that $\un{w}$ and $g(\un{w})$ are two different reduced
expressions for the same element. Then any rex move $B_{\un{w}} \to B_{g(\un{w})}$ induces the canonical isomorphism $B_w \to g(B_w)$. Together, the collection of these canonical
isomorphisms places the \emph{standard} equivariant structure on $B_w$, which we denote as $(B_w, +1)$. Any other equivariant structure has the form $(B_w \sqot V)$ for some representation $V$ of the stabilizer $G_w$.

One immediate consequence is the following.

\begin{claim} Let $\un{w} = s_1 s_2 \ldots$ denote a reduced expression in $W$, for an element $w \in W \subset W'$. Then $(B_w,+1)$ is the top summand of $(B_{s_1},+1) \ot (B_{s_2},+1)
\ot \cdots$ in $\DC_G$. \end{claim}

Warning: $s_i$ is not a simple reflection in $W'$, but instead a product of commuting simple reflections $I(s_i)$. The indecomposable $B_{s_i}$ is in fact isomorphic to $B_{t_1} B_{t_2} \cdots B_{t_{L(s_i)}}$ for any order on the simple reflections $t_j \in R(s_i)$.

\begin{proof} We know that $B_w$ is the top summand of $B_{s_1} \ot B_{s_2} \ot \cdots$ inside $\DC$, and the space of possible idempotents is one-dimensional. The equivariant structure
map of $(B_{s_1},+1) \ot (B_{s_2},+1) \ot \ldots$ is a tensor product of rex moves, and is therefore itself a rex move. This induces the standard equivariant structure on $B_w$.
\end{proof}

Our goal is to compute the weighted Grothendieck ring $[\DC]_\si$ for $\si \in G$. We will not be able to do this in general, though we can do it when $G$ is abelian and $\si \in G$
acts transitively on each orbit. In this case, the reductions of section \ref{subsec-reductions} allow us to assume that $G$ is cyclic, generated by $\si$. The results of section
\ref{subsec-abelian} imply that $[\DC]_\si$ is a free $\Zvv$-module with basis given by $[B_w,+1]$ for $w \in (W')^G = W$.

\begin{cor} Suppose that $G$ is abelian, and $\si \in G$ is such that $\{ [B_w,+1] \}_{w \in W \subset W'}$ span $[\DC]_\si$. Then the elements $\{ [B_s,+1] \}_{s \in S}$ generate
$[\DC]_\si$ as a ring. \end{cor}

\begin{proof} By the previous claim, $(B_w,+1)$ is a summand of a tensor product of various $(B_s,+1)$, having multiplicity one. The remaining summands are $(B_v \sqot V)$ for $v < w$,
whose images in the Grothendieck group are in the span of $[B_v,+1]$. Therefore, the fact that each $[B_w,+1]$ is in the subring generated by $[B_s,+1]$ is a simple upper-triangularity
argument. \end{proof}

\begin{thm} \label{thm:main} Let $G$ act on $(W',S')$, and let $(W,S,L)$ be the corresponding quasi-split embedded Coxeter subgroup. Suppose $G$ is abelian, and $\si \in G$ acts
transitively on each $G$-orbit in $S'$. Then there is an isomorphism of $\Zvv$-algebras $\HC(W,S,L) \to [\DC]_\si$, sending $b_w \mapsto [B_w,+1]$. The weighted character map (defined
in \S\ref{subsec-wtchar}) provides an inverse isomorphism. \end{thm}

\begin{proof} The proof will take place throughout the remainder of this chapter, but we outline it here.

First we show the existence of an algebra map $\psi \co \HC(W,S,L) \to [\DC]_\si$, sending $b_s \mapsto [B_s,+1]$. To show this, we must check the relations of the presentation given in
section \ref{subsec-KLpresentation}. This is really the crux of the proof, and follows from the explorations into decomposition of the following sections.

The remainder of the proof is an abstract and mostly tautological unwinding of the consequences of the categorification results in section \ref{subsec-catfnthms}. By the previous
corollary, there is a unipotent upper triangular change of basis matrix between certain monomials in $[B_s,+1]$, and $\{[B_w,+1]\}$. Since $[B_s,+1]$ generates $[\DC]_\si$ as a ring, so
the map $\psi$ is surjective. As another consequence, the change of basis matrix between $\psi(b_w)$ and $[B_w,+1]$ must also be unipotent upper triangular, and thus $\psi$ is injective.

In the final section we define the weighted character map, and claim that it expresses $[M,\phi]$ in terms of the standard basis of $\HC(W,S,L)$. Using this, one can show that
$[B_w,+1]$ satisfies the defining properties of the KL basis element $b_w$. However, we only sketch this result, as our favorite proof requires a much more in-depth discussion of
localization and light leaves in order to state rigorously. \end{proof}

Let $G$ be abelian, and suppose that $\si \in G$ does not act transitively on each orbit. It is easy to find an element $x \in W' \setminus W$ fixed by $\si$, and the direct sum $B_{Gx}$
will contribute to the $\si$-weighted Grothendieck group of $\DC$. Therefore, it is not possible that $[\DC]_\si$ is the Hecke algebra with unequal parameters, because it has a basis
strictly larger than $W$. Nonetheless, when we perform our decompositions below, we will not assume that $\si$ acts transitively on each orbit. We will still be able to compute certain
formulas in the Grothendieck ring, though this will not be enough to determine the entire ring structure. Some non-abelian examples are also computable in this partial sense.

For standard Coxeter embeddings, the computations below are done mostly in full. For the embedding of $A_1$ in $I_2(m)$, or $B_2$ in $A_4$, the computations are not hard but use the
\emph{thick calculus} of \cite{EThick} and \cite{ECathedral}. In order to keep this paper shorter, we just sketch these computations. We have not attempted the computation for the
embedding of $I_2(8)$ inside $F_4$.

\subsection{Checking the quadratic relation I}
\label{subsec-quadraticI}

We begin by checking the quadratic relation \eqref{eq:bb}. This calculation will provide an excellent illustration of the methods involved.

First, we remind the reader why \eqref{eq:bb} holds in the split case. Suppose that $s$ is a simple reflection. Then \begin{equation} \label{eq:BB} B_s B_s \cong B_s(1) \oplus B_s(-1)
\end{equation} in $\DC$. To prove this, one constructs inclusion maps $i_1, i_2 \co B_s \to B_s B_s$ of degrees $+1, -1$ respectively, and projection maps $p_1, p_2 \co B_s B_s \to B_s$
of degrees $-1, +1$. These maps are pictured below, and they satisfy: \begin{subequations} \begin{equation} 1_{B_s B_s} = i_1 p_1 + i_2 p_2, \end{equation} \begin{equation} 1_{B_s} =
p_1 i_1 = p_2 i_2, \end{equation} \begin{equation} 0 = p_1 i_2 = p_2 i_1. \end{equation} \end{subequations}

\begin{center} ${
\labellist
\small\hair 2pt
 \pinlabel {$i_1 = \frac{1}{2}$} [ ] at -20 64
 \pinlabel {$i_2 =$} [ ] at 58 64
 \pinlabel {$p_1 =$} [ ] at -20 18
 \pinlabel {$p_2 = \frac{1}{2}$} [ ] at 58 18
\endlabellist
\centering
\ig{1}{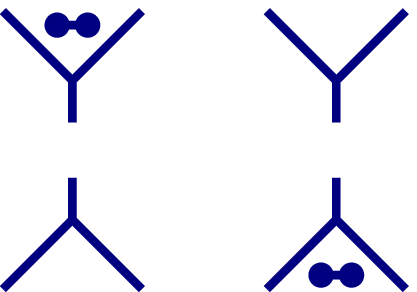}
}$\end{center}

As an extended warmup, let us consider the diagonal quasi-split embedding of $A_1$ into $A_1 \times A_1$, with $G = \ZM/2\ZM$, so that $\si$ switches the two colors in $A_1 \times A_1$.
The reader should keep Example \ref{compare A1A1} in mind. Let $S' = \{t,u\}$ and let $S = \{s\}$, with the embedding $s=tu$. Viewing $s$ as an element of $W'$ one has $B_s = B_t B_u$.
By the pulling apart relation (see Remark \ref{remark distant}), one has $B_t B_u \cong B_u B_t$. Therefore, in $\DC$ one has \begin{equation} B_s B_s \cong B_t B_u B_t B_u \cong B_t
B_t B_u B_u \cong B_s(2) \oplus {\color{red} B_s \oplus B_s} \oplus B_s(-2).\end{equation} Thus we can write the identity of $B_t B_u B_t B_u$ as a sum of four orthogonal idempotents,
as below. These factor with inclusion and projection maps $i_k, p_k$, for $1 \le k \le 4$.

\begin{center} ${
\labellist
\small\hair 2pt
 \pinlabel {$i_1 = $} [ ] at -20 231
 \pinlabel {$p_1 = $} [ ] at -20 72
 \pinlabel {$i_2 = \frac{1}{2\sqrt{2}}$} [ ] at 78 281
 \pinlabel {$+ \frac{1}{2\sqrt{2}}$} [ ] at 182 283
 \pinlabel {$p_2 = \frac{1}{2\sqrt{2}}$} [ ] at 77 120
 \pinlabel {$+ \frac{1}{2\sqrt{2}}$} [ ] at 180 120
 \pinlabel {$i_3 = \frac{1}{2\sqrt{2}}$} [ ] at 76 189
 \pinlabel {$- \frac{1}{2\sqrt{2}}$} [ ] at 181 185
 \pinlabel {$p_3 = - \frac{1}{2\sqrt{2}}$} [ ] at 74 24
 \pinlabel {$+ \frac{1}{2\sqrt{2}}$} [ ] at 181 23
 \pinlabel {$i_4 = $} [ ] at 281 232
 \pinlabel {$p_4 = $} [ ] at 278 69
\endlabellist
\centering
\ig{1}{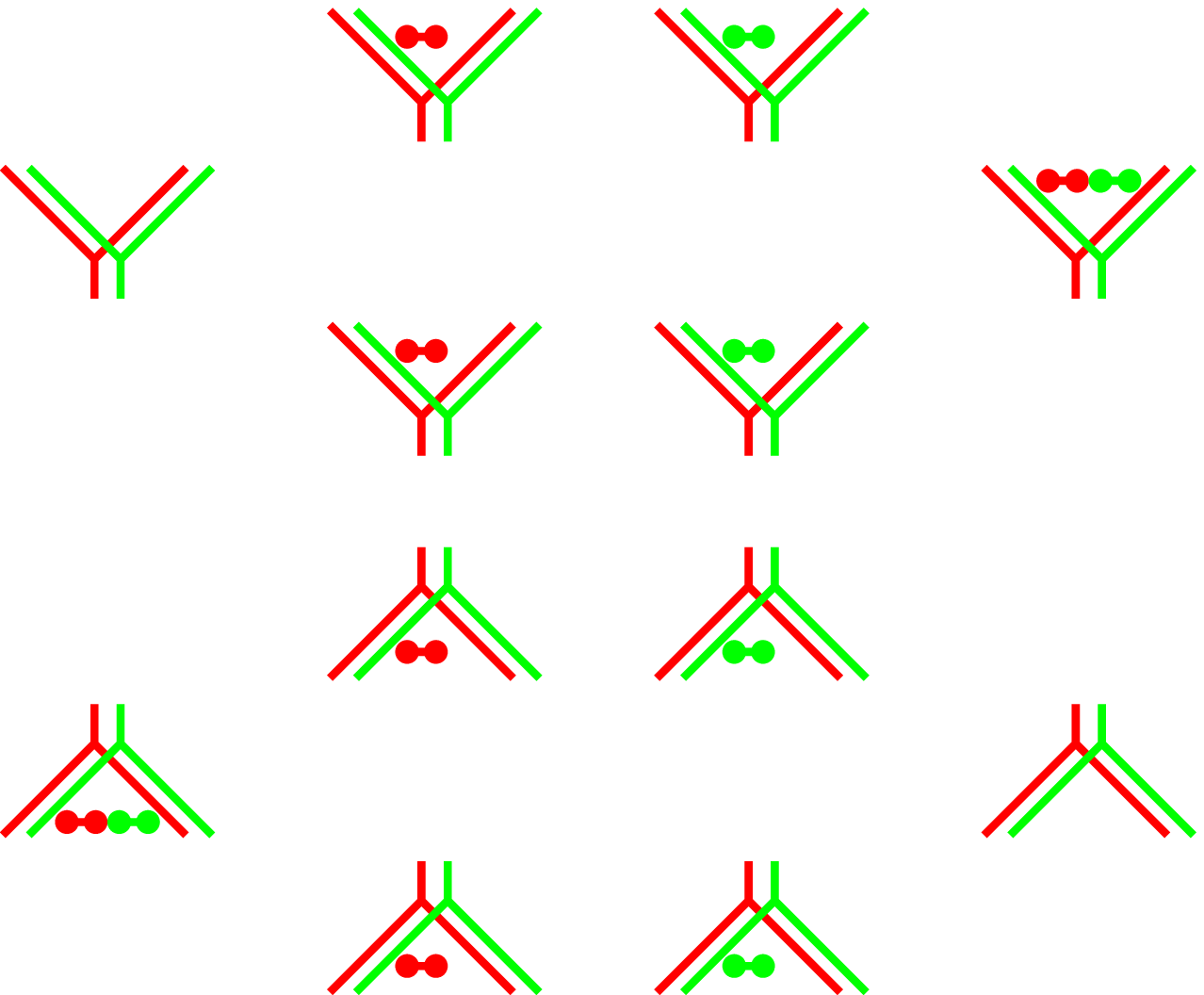}
}
$ \end{center}

Now we ask what happens in the equivariant category $\DC_G$. The equivariant structure map $\phi_\si$ of $(B_t B_u,+1)$ is the 4-valent vertex $\ig{1}{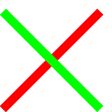}$. Then $(B_t B_u, +1) \ot
(B_t B_u, +1)$ has equivariant structure map $\ig{1}{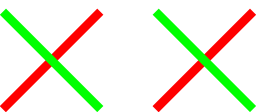}$. Using distant-isotopy one can observe that $$\ig{1}{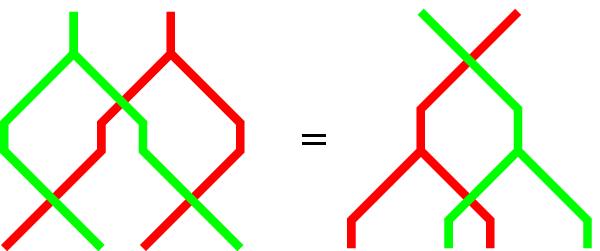}.$$ Symbolically, this says that $\si(p_4) \circ \phi
= \phi \circ p_4$, which is precisely the condition that $p_4$ is a map from $(B_t B_u,+1) \ot (B_t B_u,+1)$ to $(B_t B_u,+1)$. However, the reader can observe that $p_3$ and $i_3$ are
maps between $(B_t B_u,+1) \ot (B_t B_u, +1)$ and $(B_t B_u,-1)$, as $\si(p_3) \circ \phi = - \phi \circ p_3$.

Together, these maps imply the direct sum decomposition \begin{align} \label{eq:A1A1equivdecomp} (B_t B_u, +1) & \ot (B_t B_u,+1) \cong \\ \nonumber & (B_t B_u (2),+1) \oplus
{\color{red}(B_t B_u(0),+1) \oplus (B_t B_u(0),-1)} \oplus (B_t B_u(-2),+1) \end{align} in $\DC_G$. In the weighted Grothendieck group $[\DC]_\si$, the two terms with no degree shift
cancel, leaving $[B_s,+1]^2 = (v^2 + v^{-2})[B_s,+1]$ as desired.

The decomposition \eqref{eq:A1A1equivdecomp} holds true for any $G$ acting transitively on $A_1 \times A_1$, not just when $G = \ZM/2\ZM$. Let $G$ be a group acting transitively on $A_1
\times A_1$, and let $K$ be the kernel of the action, so that $G/K \cong \ZM/2\ZM$. Let $-1$ denote the representation of $G$ which factors through the sign representation of $G/K$. Then
$G$ is the stabilizer of the element $s = tu \in W'$, and it is clear that \eqref{eq:A1A1equivdecomp} holds, interpreting $+1$ and $-1$ as representations of $G$.

When $G$ is abelian, one can compute directly in the weighted Grothendieck group without needing to construct the specific inclusion and projection maps, instead using Corollary
\ref{cor:trace}. One can examine the graded vector space $V^k$ of (linear combinations of) morphisms $B_t B_u B_t B_u \to B_t B_u$ in degree $k$, modulo the kernel of the LIP. The $G$ action of \eqref{Gxaction} is given by:

\begin{equation} \label{eq:actionexample} {
\labellist
\small\hair 2pt
 \pinlabel {$f$} [ ] at 27 52
 \pinlabel {$\mapsto$} [ ] at 73 52
 \pinlabel {$\si(f)$} [ ] at 120 52
\endlabellist
\centering
\ig{1}{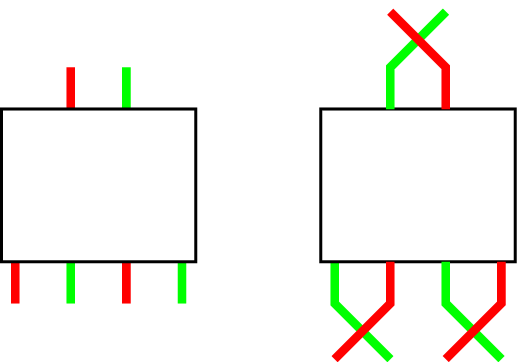}
}\end{equation}

Computing the trace of $\si$ will give multiplicities in the Grothendieck group $[\DC]_\si$. For instance, $V^{-2}$ one-dimensional, spanned by $p_4$. Since $\si$ acts by the identity on
this space, $[B_t B_u(2),+1]$ appears with multiplicity 1. On the other hand, $V^0$ is spanned by the two pictures in either $p_2$ or $p_3$, and these two pictures are switched by $\si$,
so the trace is zero. Therefore $[B_t B_u(0),+1]$ appears with multiplicity zero.

\begin{remark} The actual space of degree $0$ morphisms is larger than two-dimensional, also allowing for positive degree polynomials to be placed next to $i_1$. However, all such maps
are in the kernel of the LIP, as is any map in the image of the left (or right) action of positive degree polynomials. \end{remark}

\subsection{Checking the quadratic relation II}
\label{subsec-quadraticII}

More generally, $A_1$ can embed in a quasi-split embedding as the longest element of $A_1^k$, or as the longest element of $I_2(m)^k$.

First we treat the case of $A_1^k$, where $k = L(s)$. Let $s \in S$ be the product of the commuting reflections $I(s)=\{t_1,\ldots, t_{L(s)}\} \subset S'$. One has $B_s = B_{t_1} \cdots
B_{t_{L(s)}}$. Our goal is to decompose $(B_s, +1) \ot (B_s,+1)$ into indecomposables in $\DC_G$. Because of the reductions in section \ref{subsec-reductions}, we may work within the
subcategory of $\DC$ which only involves the colors in $I(s)$, and we may assume that $G$ acts on $I(s)$ faithfully, so that we identify $G$ with a subgroup of the symmetric group
$S_{L(s)}$. We can make this assumption even when $G$ is not abelian; however, translating this decomposition into a statement about the weighted Grothendieck ring will be difficult
unless $G$ is abelian.

In $\DC$, the only summands of $B_s B_s$ are grading shifts of $B_s$. We already know what the graded vector space $V^{\bullet}(B_s B_s, B_s)$ of projection maps is: it has a basis given
by maps of the following form. \igc{1}{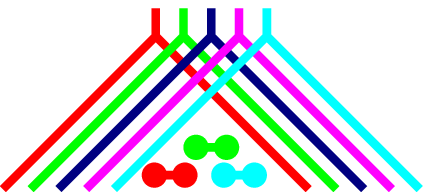} There is one such map for each subset $J \subset I(s)$ indicating which colors of barbells appear, and the map has degree
$-L(s)+2|J|$. The group $G$ acts on this basis precisely as it acts on the subsets of $I(s)$, so that we need only decompose this subset representation into irreducibles for $G$.

If $G$ is abelian, then computing the product in the $\si$-weighted Grothendieck group is easy. One has \begin{equation} [B_s,+1]^2 = \sum_{i = 0}^{L(s)} v^{-L(s) + 2i} c_i [B_s,+1],
\end{equation} where $c_i$ is the number of subsets $J \subset I(s)$ which are fixed by $\si$, with $|J|=i$. Thinking of $\si$ as an element in the symmetric group $S_{L(s)}$, the cycle
decomposition of $\si$ will determine these multiplicities $c_i$. For instance, if $L(s)=3$ and $\si$ is a 2-cycle, then $[B_s,+1]^2 = (v^3 + v^1 + v^{-1} + v^{-3}) [B_s,+1]$. Similarly, if
$L(s)=4$ and $\si$ is a product of disjoint 2-cycles then $[B_s,+1]^2 = (v^4 + 2 + v^{-4}) [B_s,+1]$. In particular, one has \begin{equation} [B_s,+1]^2 = (v^{L(s)} + v^{-L(s)}) [B_s,+1]
\end{equation} if and only if $\si$ is a cycle with no fixed points. This was the quadratic relation in $\HC(W,S,L)$.

When $\si$ is a cycle, only the idempotents where each color behaved the same will contribute to the weighted Grothendieck group. \emph{Mixed idempotents}, where some colors behave
in one way, and others in another way, may appear as summands in $\DC_G$ but will not contribute to the weighted Grothendieck group. This is because mixed idempotents are permuted by
$\si$ with no fixed points.

Unraveling our reductions, one has the following statement: let $G$ be abelian, acting (possibly non-faithfully) on $(W',S',\ell)$, and choose $\si \in G$. Suppose the $G$-split
embedding is standard. Then the quadratic relation in $\HC(W,S,L)$ holds for $[B_s,+1]$ in $[\DC]_\si$ if and only if $\si$ acts transitively on the orbit $I(s)$.

Now we sketch what happens for the embedding of $W = A_1$ into $W' = I_2(m)$, sending $s$ to the longest element $w_0$ of the dihedral group. Let $\si \in \ZM / 2 \ZM$ swap the two
colors of the dihedral group. We know that $B_{w_0} \ot B_{w_0} \cong [W'] B_{w_0}$, where again $[W']$ denotes the balanced Poincar\'e polynomial of the dihedral group. A choice of
inclusion and projection maps can be made by choosing dual bases for the Frobenius extension $R^{W'} \subset R$, relative to the Demazure operator $\pa_{w_0}$. For more information on this decomposition, see the discussion of thick calculus in \cite[\S 6.3]{ECathedral}, and perhaps also the diagrammatic discussion surrounding \cite[Equation (3.43)]{EThick}.

Finding dual bases is an exercise in Schubert calculus; observing that $\si$ is traceless on all but the highest and lowest degrees, where it acts as the identity, is also an exercise.
Thus one has \begin{equation} [B_s,+1]^2 = (v^m + v^{-m}) [B_s,+1] \end{equation}
inside the $\si$-weighted Grothendieck group.

The embedding of $A_1$ into $I_2(m)^k$ merely combines these different observations, but is no more difficult, and we leave it to the reader.

\subsection{Type $A_1 \times A_1$ inside type $A_1^{\times k} \times A_1^{\times l}$}
\label{subsec-A1A1inA1A1}

Now we begin checking the dihedral relations, as enumerated in Proposition \eqref{prop: quasi-split}. We wish to show that the expression given for $b_{w_0} \in \HC(W,S,L)$ in terms of
the KL generators $b_s$ is categorified by the appropriate decomposition of certain tensor products of $(B_s,+1)$.

We begin with the easiest case, the embedding of $A_1 \times A_1$ inside $A_1^{\times k} \times A_1^{\times l}$. The reader may consult example \ref{compare A1^k A1^l} and
\eqref{eq:dihedral m=2}. Let $S = \{s,t\}$. We seek to show that \begin{equation} (B_s,+1) \ot (B_t,+1) \cong (B_{st},+1), \end{equation} regardless of the group $G$. Certainly, one
already knows that $B_s B_t \cong B_{st}$ in $\DC$.

For any $g \in G$, the corresponding equivariant structure map $\phi_g \co B_s \to g(B_s)$ inside $(B_s,+1)$ is just a rex move, as is the map $\phi_g \co B_t \to g(B_t)$. Their tensor product is also a rex move. Therefore, the induced equivariant structure on the unique summand $B_{st}$ is $(B_{st},+1)$.

As an example, let $R(s)=\{u_1,\ldots,u_k\}$ and $R(t) = \{v_1,\ldots,v_l\}$, and suppose that $g$ acts on each orbit cyclically and transitively. We may assume that $g$ sends $u_i$ to
$u_{i+1}$ modulo $k$, and sends $v_i$ to $v_{i+1}$ modulo $l$. Choosing an isomorphism $B_s \cong B_{u_1} B_{u_2} \cdots B_{u_k}$, the equivariant structure map $\phi_g$ for $(B_s,+1)$
is \igc{1}{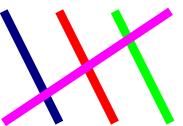} The map for $(B_t,+1)$ is similar. Therefore, the equivariant structure map on $B_s B_t$ induced by $(B_s,+1) \ot (B_t,+1)$ is \igc{1}{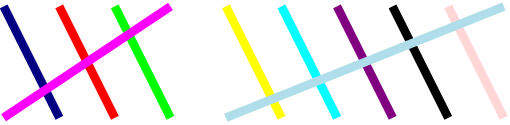} This
is a rex move, and therefore it agrees with the standard equivariant structure map for $B_{st}$.

\subsection{Type $B_2$ inside type $A_3$}
\label{subsec-B2inA3}

Now we consider the embedding of $B_2$ inside $A_3$. The reader may consult Example \ref{compare A3B2}, and the $k=1$ case of \eqref{eq:dihedral m=4}. For purposes of computing
decompositions we can assume that $G$ acts on $A_3$ faithfully with the appropriate orbits, so that $G = \ZM/2\ZM$. Let $\si \in G$ be the generator.

Let $S=\{s,t\}$ and $S' = \{x,y,z\}$, and let $\si \in G$ switch $x$ and $z$. Let $s=y$ and $t=xz$. Inside $\DC$ one has the decomposition \begin{equation} B_y B_x B_z B_y B_x B_z \cong
B_{yxzyxz} \oplus {\color{blue}(B_{yxyz} \oplus B_{yzyx})} \oplus B_{yxz}(1) \oplus B_{yxz}(-1), \end{equation} as can be easily computed using the Hecke algebra $\HC(W',S',\ell)$. It remains to
compute what happens in $\DC_G$.

The equivariant structure map of $\si$ on $(B_s,+1) \ot (B_t,+1) \ot (B_s,+1) \ot (B_t,+1)$ is \igc{1}{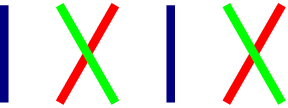} As discussed previously, this induces the equivariant structure
$(B_{stst},+1)$ on the top summand. It induces the only possible equivariant structure on the summand $B_{yxyz} \oplus B_{yzyx}$. We wish to show that the equivariant structures induced
on $B_{st}(1)$ and $B_{st}(-1)$ are the standard ones.

The projection to $B_{st}(-1)$ is (up to scalar) the map \igc{1}{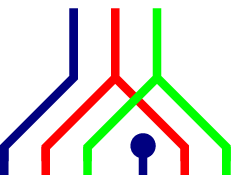}
and its possible inclusions have the form
{
\labellist
\small\hair 2pt
 \pinlabel {$+ f$} [ ] at 82 22
\endlabellist
\centering
\igc{1}{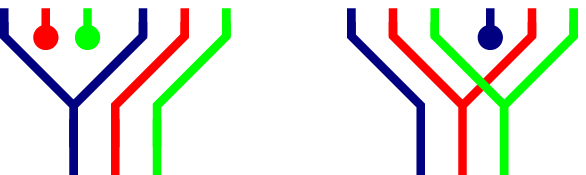}
}
for some linear polynomial $f$. Placing these maps upside down one obtains the possible projections and inclusions for $B_{st}(1)$, except with some polynomial $g$ instead of $f$. In order for the corresponding idempotents to be
orthogonal, the required condition is that $f + g = -\pa_y(\al_x + \al_z)$. For instance, we may let $f=0$ for the inclusion of $B_{st}(-1)$, and
$g=-\pa_y(\al_x+\al_z)$ for the projection to $B_{st}(1)$. Both polynomials are invariant under the action of $\si$. It is easy to see that these idempotents are acted on trivially by $\si$, in the sense of \eqref{eq:actionexample}.

\begin{remark} One need not actually find the two orthogonal idempotents to perform this computation. After all, the terms with polynomials above are in the kernel of the LIP, and can
effectively be ignored. Also, by duality, one only need consider the trace of $\si$ on projection maps of degree $-1$, and the trace on projection maps of degree $+1$ will be determined.
\end{remark}

Therefore, in $\DC_G$ one has the decomposition \begin{align} \label{eq:decompB2inA3} (B_s,+1) \ot (B_t,+1) & \ot (B_s,+1) \ot (B_t,+1) \cong \\ \nonumber & (B_{stst},+1) \oplus
{\color{blue}(B_{yxyz} \oplus B_{yzyx},\phi)} \oplus (B_{st}(1),+1) \oplus (B_{st}(-1),+1). \end{align} This is true for all $G$, not just $\ZM/2\ZM$. When $G$ is abelian and $\si$ descends to a
generator of $\ZM/2\ZM$, this categorifies \eqref{eq:dihedral m=4 alt} in the $\si$-weighted Grothendieck group.

\subsection{Type $B_2$ inside type $A_4$}
\label{subsec-B2inA4}

The reader should consult Example \ref{compareA4B2}, and we continue to use its notation. Let $S = \{t,u\}$ have type $B_2$ and $S' = \{s_1,s_2,s_3,s_4\}$ have type $A_4$, and set $t =
s_1 s_4$ and $u = s_2 s_3 s_2$. We already know from Example \ref{compareA4B2} that \begin{align} B_t & \ot B_u \ot B_t \ot B_u \cong B_{tutu} \oplus B_{tu}(1) \oplus B_{tu}(-1) \\
\nonumber & \oplus {\color{blue} B_{s_2 s_1 w_{234}} \oplus B_{s_3 s_4 w_{123}} \oplus B_{s_1 w_{234}}(1) \oplus B_{s_1 w_{234}}(-1) \oplus B_{s_4 w_{123}} (1) \oplus B_{s_4
w_{123}}(-1)}. \end{align} If $\si$ is the usual automorphism of $S'$, then the blue terms have a unique equivariant structure for $\si$, and will not contribute to the $\si$-weighted
Grothendieck group when $\si$ lives in an abelian group $G$. As in the previous section, we need only check that the projection map from $B_t B_u B_t B_u$ to $B_{tu}(-1)$ is acted on
trivially by $\si$.

In the thick calculus of \cite{EThick}, this projection map can be drawn as follows.
\igc{1}{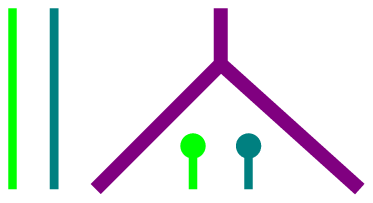}
The thick trivalent vertex has degree $-3$, so the overall diagram has degree $-1$. It is evidently fixed by $\si$.

\subsection{Diagonal embeddings}
\label{subsec-diagonal}

In section \ref{subsec-B2inA3} we treated the quasi-split embedding of $B_2$ inside $A_3$, though we did not treat the embedding of $B_2$ inside $A_3^{\times k}$. Similarly, we have not
yet treated the case of the diagonal embedding of $I_2(m)$ inside $I_2(m)^{\times k}$. In this section, we demonstrate by example that such multiplicities do not further complicate the
computations in the $\si$-weighted Grothendieck group, when $G$ is abelian and $\si$ acts transitively on the orbits; the essential difficulties which arise have already been addressed in
section \ref{subsec-A1A1inA1A1}. The reader should be able to formulate the general argument from this example.

Consider the embedding of $B_2$ inside $A_3^{\times 3}$. Let $S=\{s,t\}$, with $L(s)=3$ and $L(t)=6$. Color the three elements of $R(s)$ by different shades of blue: light, medium and
dark. Do the same for the elements of $R(t)$, and different shades of red and green. The dark colors commute with the light and medium colors, and so forth. The equivariant structure map
on $(B_s,+1) \ot (B_t,+1) \ot (B_s,+1) \ot (B_t,+1)$ is \igc{1}{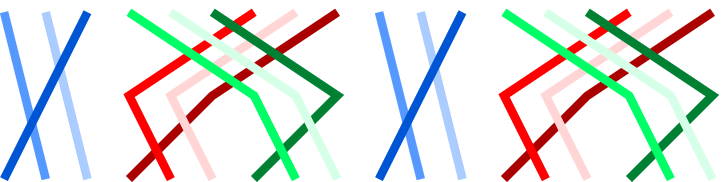}

Given any Soergel diagram for $A_3$, one can obtain a Soergel diagram for $A_3^{\times 3}$ by replacing each blue strand with three shades of blue strands, and so forth. This operation
will triple the degree of the map, but will not change how the maps compose thanks to the principle of distant-isotopy (see Remark \ref{remark distant}). Thus we still have idempotents
inside $(B_s,+1) \ot (B_t,+1) \ot (B_s,+1) \ot (B_t,+1)$ for each summand mentioned in \eqref{eq:decompB2inA3}, except that the degree shifts are tripled.

However, there will also be a host of mixed idempotents. For example, one may project to $B_{st}(1)$ in the dark shade, while projecting to $B_{st}(-1)$ in the medium and light shades.
Here is such a projection map: \igc{1}{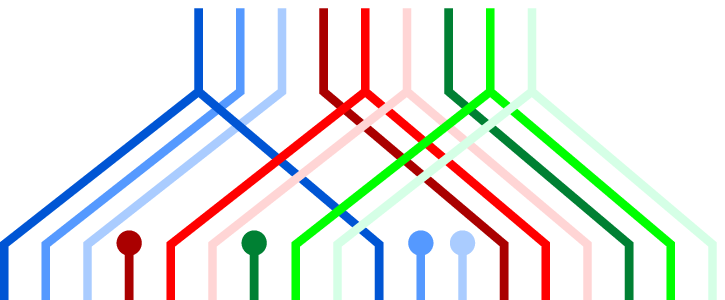} Just as in section \ref{subsec-A1A1inA1A1}, such mixed idempotents are permuted by $\si$ with no fixed points, so that they do not
contribute to the $\si$-weighted Grothendieck group.

Computing the full decomposition of $(B_s,+1) \ot (B_t,+1) \ot (B_s,+1) \ot (B_t,+1)$ in this and more general situations seems to be a tedious problem.

\subsection{Type $G_2$ inside type $D_4$}
\label{subsec-G2inD4}

Finally, we treat the embedding of $G_2$ inside $D_4$. The embedding of $G_2$ inside $D_4^{\times k}$ is treated in similar fashion to the preceding section. The reader may consult
Example \ref{compare D4G2}, and the $k=1$ case of \eqref{eq:dihedral m=6}. Let $S=\{s,t\}$ and $S' = \{u_1,u_2,u_3,v\}$. As above, we assume that $G$ acts faithfully on $D_4$, thus
embedding into the symmetric group on $\{u_1,u_2,u_3\}$. Since $G$ acts transitively, it is either $\ZM/3\ZM$ or $S_3$. Either way there is some element $\si \in G$ which cyclically
permutes the $u_i$. Let $s=v$ and $t=u_1u_2u_3$.

To shorten our notation, we denote the bimodule $B_s B_t B_s B_t B_s B_t$ as $X$. In this case, the indecomposables $B_s$ and $B_t$ are no more than products of the generators, so that
$X$ corresponds to the sequence $\un{v u_1 u_2 u_3 v u_1 u_2 u_3 v u_1 u_2 u_3}$, an object of $\DC$, and not some more complicated object in the Karoubi envelope. We write $(X,+1)$ for
the equivariant object induced as a tensor product. We also give names to several of the representations of $G$ or stabilizers inside $G$. Let $V_{\std}$ denote the standard
3-dimensional representation of $S_3$, and let $V_{\reg}$ denote the 6-dimensional regular representation; when $G = \ZM/3\ZM$ we also let these represent the restriction of these
$S_3$-representations to the subgroup $G$. Let $+1$ denote the trivial representation of the stabilizer of $u_1$ in $G$, whether this stabilizer be $S_2$ or the trivial group.

We already know the decomposition of $X$ in $\DC$ for $D_4$, given on the decategorified level in Example \ref{compare D4G2}. We reprint the equality here.

\begin{align*} b_s b_t b_s b_t b_s b_t = b_{ststst} + {\color{blue} (b_{u_1 u_2 stst} + b_{u_1 u_3 stst} + b_{u_2 u_3 stst})} + {\color{blue} (v+v^{-1}) (b_{u_1 stst} + b_{u_2 stst} +
b_{u_3 stst})} \\ + (2v^2 + {\color{red} 6} + 2v^{-2}) b_{stst} + {\color{blue} (v+v^{-1}) (b_{vu_1 u_2 st} + b_{vu_1 u_3 st} + b_{v u_2 u_3 st})}\\ + (v^{4} + {\color{red} 6v^2} +
(3+{\color{red} 3} + {\color{green} 6}) + {\color{red} 6v^{-2}} + v^{-4}) b_{st}. \end{align*}

Terms which are blue correspond to summands whose orbits have size $3$; there is a unique equivariant structure when $G = \ZM/3\ZM$, and two possible equivariant structures when $G =
S_3$. Such terms can not possibly contribute the $\si$-weighted Grothendieck group when $G$ is abelian, because the indecomposables are not $\si$-fixed. We seek to show that the terms which
are black have the standard equivariant structure, and the terms which are red correspond to summands with equivariant structure $V_{\std}$, and the green term corresponds to a summand
with equivariant structure $V_{\reg}$. The red and green terms would then cancel in the $\si$-weighted Grothendieck group when $G$ is abelian, since the trace of $\si$ on $V_{\std}$ or
$V_{\reg}$ is zero.

\begin{prop} The following direct sum decomposition holds in $\DC_G$. We write polynomials to indicate direct sums of grading shifts.

\begin{align} (&X,+1)  \cong (B_{ststst},+1) \oplus {\color{blue} (B_{u_1 u_2 stst} \oplus B_{u_1 u_3 stst} \oplus B_{u_2 u_3 stst},+1)} \\ \nonumber & \oplus {\color{blue} (v+v^{-1})
(B_{u_1 stst} \oplus B_{u_2 stst} \oplus B_{u_3 stst},+1)} \oplus (2v^2 + 2v^{-2}) (B_{stst},+1) \oplus {\color{red} 2 (B_{stst} \sqot V_{\std})} \\ \nonumber & \oplus {\color{blue}
(v+v^{-1}) (B_{vu_1 u_2 st} + B_{vu_1 u_3 st} + B_{v u_2 u_3 st},+1)} \\ \nonumber & \oplus (v^{4} + v^{-4}) (B_{st},+1) \oplus {\color{red} (2v^2+2v^{-2}) (B_{st} \sqot V_{\std})}
\oplus 3 (B_{st},+1) \oplus {\color{red} (B_{st} \sqot V_{\std})} \oplus {\color{green} (B_{st} \sqot V_{\reg})}. \end{align}

\end{prop}

Let us discuss the algorithm by which we compute the spaces $V^k(X,B_w)$, for various $k$ and $w$, in order to determine the action of the stabilizer $G_w$ on them. By duality, we need
only investigate the case $k \le 0$. The algorithm will proceed degreewise, from the most negative $k$ up to zero, for a reason which will become clear soon.

Clearly $\be \co \Hom^k(X,\un{w}) \to V^k(X,B_w)$ is a surjection. First, we find a set of linearly independent elements of $\Hom^k(X,\un{w})$, for a reduced expression of $w$,
which descend to a spanning set of $V^k(X,B_w)$. The light leaves of degree exactly $k$ serve this purpose. This is because double leaves form a basis of $\Hom(X,\un{w})$ as a right
$R$-module; any map in the image of $R_+$ acting on the right will be in the kernel of $\be$, and any double leaf which is not a light leaf followed by the identity map will factor through
lower terms, and thus be in the kernel as well. Thus we actually think of $\be$ as a map from $Y^k(X,w)$, the span of light leaves of degree $k$, to $V^k(X,B_w)$.

Next, we compute the kernel of $\be$, i.e. the kernel of the Local Intersection Pairing. It is an extrapolation of the Soergel conjecture that any morphism of non-positive degree must have
a ``reason" to be in the kernel; that the kernel is generated by maps from $X$ to $B_x$ of even more negative degree $k' < k$, followed by a map $B_x \to B_w$ of positive degree. We have
already restricted to light leaves to ignore lower terms, so that we may assume $x > w$. This map from $X$ to $B_x$ of degree $k'$ indicates that $B_x(k')$ has already occurred as a
summand of $X$ (i.e. we have already computed $V^{k'}(X,B_x)$ in some sense).

Working with objects in $\DC$, rather than indecomposables in its Karoubi envelope, we find the kernel of $\be$ by
looking at light leaves maps from $X$ to $\un{x}$ of degree $<k$, and composing with a map $\un{x} \to \un{w}$ of positive degree which is itself in the kernel of the LIP (i.e. it
has nothing of negative degree to pair against). This process will become clear in practice below.

Consider $V^{-4}(X, B_{st})$. The space of light leaves $Y^{-4}(X,st)$ is one-dimensional, spanned by \igc{1}{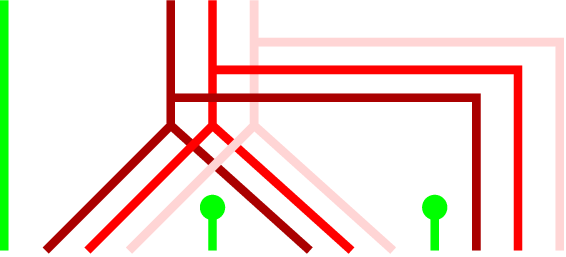} There is no kernel for $\be$, because
there are no maps of degree $<-4$ to any $B_w$ with $w > st$; if there were such a map, it must lead to a summand, and there are no such summands. Thus $Y^{-4}(X,B_{st}) = V^{-4}(X,B_{st})$, or in other words, the map pictured above is a split morphism, a projection. This diagram is fixed by the action of
$G$, so that $(B_{st}(-4),+1) \oplus (B_{st}(+4),+1)$ appear as equivariant summands of $X$. This explains the black term $(v^4 + v^{-4})b_{st}$.

Consider $V^{-2}(X,B_{stst})$. The space $Y^{-2}(X,stst)$ is spanned by the two pictures below. \begin{equation} \label{eq:G2inD4proof1} \ig{1}{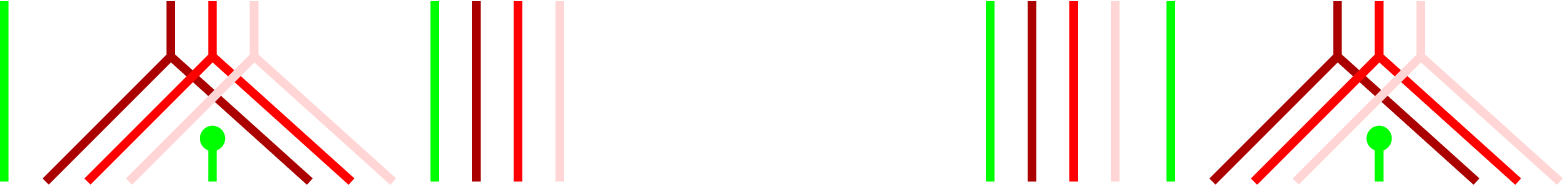} \end{equation} There is no
kernel for $\be$, because there are no summands $B_x(k)$ with $k < -2$ and $x > stst$. Thus the pictures below provide two orthogonal projection morphisms. The action of $G$ fixes both diagrams, so that one has two equivariant summands of $X$ of the form $(B_{stst}(2),+1)$, and two of the form $(B_{stst}(-2),+1)$.
This yields the black term $(2v^2 + 2v^{-2}) b_{stst}$ above.

Consider $V^{-2}(X, B_{st})$. The space $Y^{-2}(X,st)$ is six-dimensional, spanned by the $\si$-conjugates of these two light leaves. \igc{1}{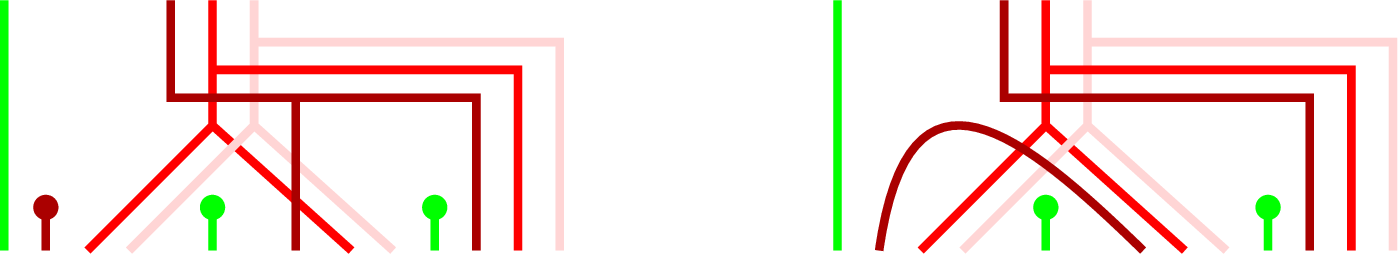} Again, there is no kernel for
$\be$, because there are no summands $B_x(k)$ with $k<-2$ and $x>st$. The action of $G$ permutes the $\si$-conjugates of each picture, so that this six-dimensional representation
is isomorphic to $V_{\std} \oplus V_{\std}$. This yields the red term $6(v^{2}+v^{-2}) b_{st}$.

Consider $V^{-1}(X, B_{u_1 stst})$. The space $Y^{-1}(X,u_1stst)$ is one-dimensional, spanned by this picture. \igc{1}{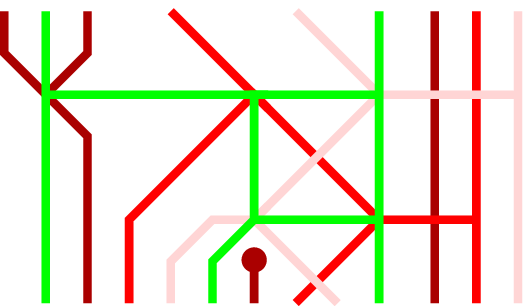} There is no kernel for degree reasons. The two other
$\si$-conjugates of this map live in different Hom spaces, and together these three projections yield the blue term $(v+v^{-1}) (b_{u_1 stst} + b_{u_2 stst} + b_{u_3 stst})$. When $G =
\ZM/3\ZM$ there is only one equivariant structure possible. When $G = S_3$ we need to determine the action of the stabilizer (switching $u_2$ and $u_3$) on this one-dimensional space. The
reader can verify that, because of the Zamolodzhikov relation \eqref{A3} applied to the central green square, the action of this stabilizer is trivial.

Consider $V^{-1}(X,B_{v u_1 u_2 st})$. The space $Y^{-1}(X,vu_1u_2st)$ is three-dimensional. Two of the three light leaves spanning this space are \igc{1}{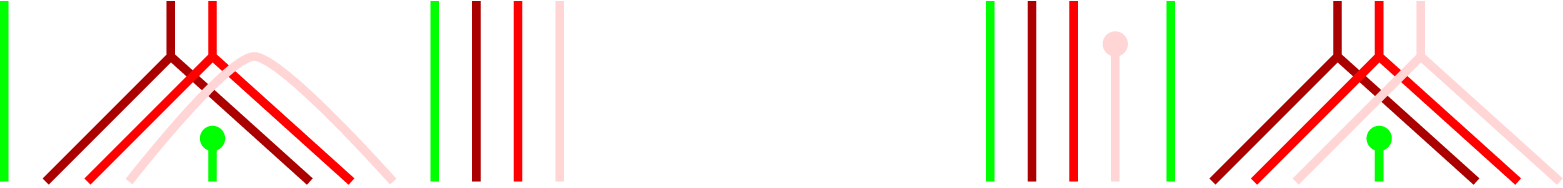} Any contribution
to the kernel of $\be$ must come from $V^{-2}(X,B_{stst})$ composed with a degree $+1$ map in the kernel of the LIP. Considering the two maps in \eqref{eq:G2inD4proof1}, and applying a dot
to the first $u_3$ strand (which must be in the kernel of the LIP, as there are no negative degree maps to balance it out), one obtains the two maps pictured here. Therefore,
$V^{-1}(X,B_{vu_1u_2st})$ is one-dimensional, spanned by the remaining light leaf map, which is \igc{1}{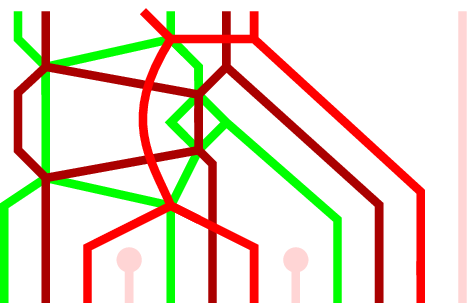} This yields the blue term $(v+v^{-1}) (b_{v u_1 u_2 st} + b_{v u_1
u_3 st} + b_{v u_2 u_3 st})$.When $G=S_3$ we need to determine the action of the stabilizer (switching $u_1$ and $u_2$) on this one-dimensional space. Applying \eqref{A3} and
\eqref{assoc2} several times one can see that this action is trivial.

Consider $V^0(X, B_{stst})$. The space $Y^0(X,stst)$ is nine-dimensional. We give three maps below; the other 6 are the $\si$-conjugates of these. \igc{1}{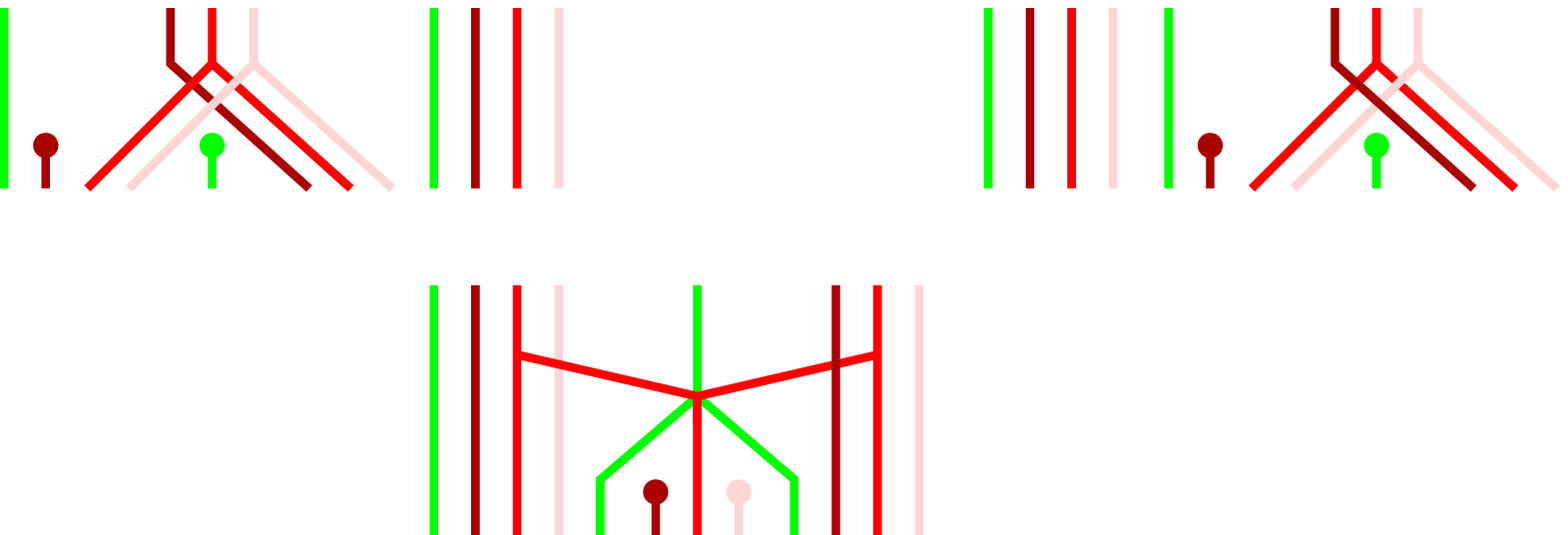} However, there is a
three-dimensional subspace in the kernel of $\be$. One element in the kernel is obtained by taking the degree $-1$ map to $\un{u_1 stst}$, and applying a dot
to the first strand. This dot is in the kernel, as no maps of negative degree exist to balance it out. \igc{1}{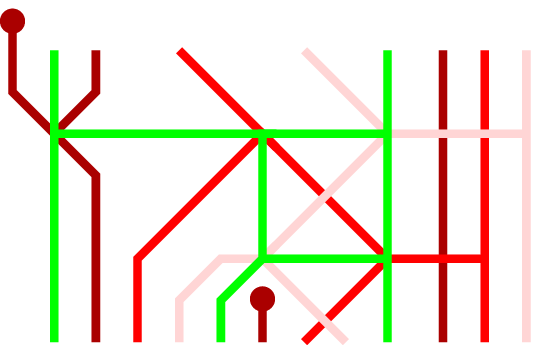} The other two generators of the kernel are the $\si$-conjugates
of this. Expanding this diagram in the nine-dimensional basis above is an exercise in the relations of $\DC$, but is not relevant here. What is obvious is that the original
nine-dimensional space gives three copies of the standard representation of $G$, and the three-dimensional kernel is a copy of the standard representation inside. What remains is a sum of
two copies of the standard representation. This yields the red term $6b_{stst}$.

Consider $V^0(X, B_{st})$. The space $Y^0(X,st)$ is eighteen-dimensional, and the kernel of $\be$ is six-dimensional. We will not bother to write down the basis of either space, which we
leave as a combinatorial exercise to the intrepid reader, but we will at least record the basis of $V^0(X,B_{st})$.

Here are the three projection maps which are $G$-fixed, giving the black term $3 b_{st}$. \igc{1}{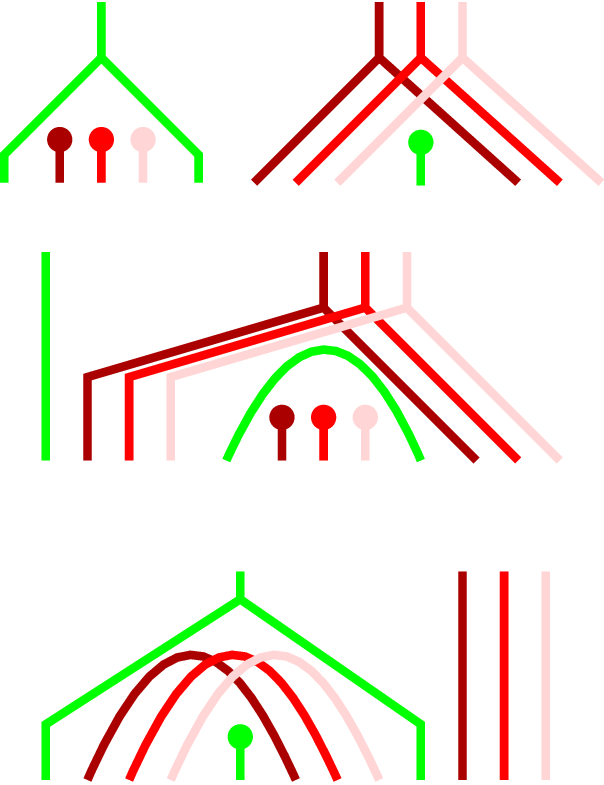} The cautious reader may worry if two of these maps are in the kernel, as
they appear to factor through the degree $-2$ maps of \eqref{eq:G2inD4proof1}, followed by this degree $+2$ map. \igc{1}{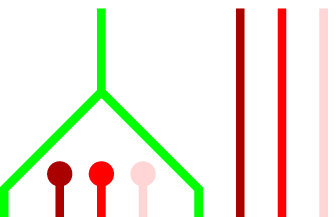} However, this degree $+2$ map is not in the kernel,
as it pairs against a map in $\Hom^{-2}(B_s B_t B_s B_t, B_{st})$. This ``appearance" of being in the kernel was an artifact of working with objects of $\DC$ rather than
indecomposables in the Karoubi envelope.

Here is a projection map which, along with its $\si$-conjugates, forms a subrepresentation isomorphic to $V_{\std}$, yielding the red term $3 b_{st}$. \igc{1}{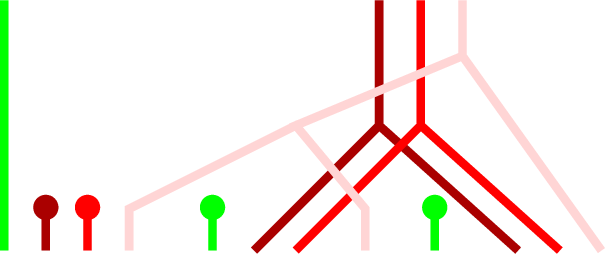}

Here is a projection map which, along with its $S_3$-conjugates, forms a subrepresentation isomorphic to $V_{\reg}$, yielding the green term $6 b_{st}$. \igc{1}{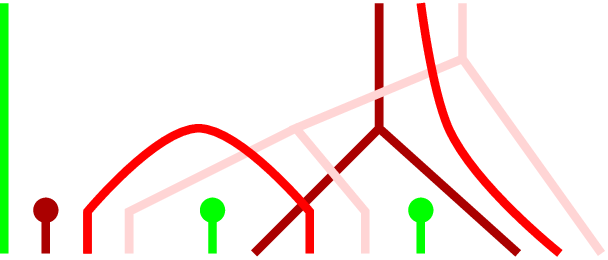}

\subsection{The weighted character map}
\label{subsec-wtchar}

We continue to assume that $G$ is abelian and that $\si$ acts transitively on the orbits of $G$.

We have finally proven that $\HC(W,S,L) \cong [\DC]_\si$. Let $(M,\phi) \in \DC_G$. There is some expression $[M,\phi] = \sum_{x \in W} P_{(M,\phi),x} H_x$ for some polynomials $P_x \in
\ZM[v,v^{-1}]$. We now reinterpret $P_x$ as the trace of $\si$ on a certain graded vector space.

The usual character map was discussed in \S\ref{subsec-catfnthms}. That is, for an object $M$ in the Karoubi envelope of $\DC$, its character in $\HC(W')$ is given as $\sum P_{w,M}
H_w$, where $P_{w,M}$ is the graded dimension of $\Hom(M,B_w)$ modulo $\DC_{< w}$. As discussed in that section, one can compute the characters of objects in $\DC$ using light leaves.

Now let $w \in W \subset W'$ be $G$-fixed, and let $(M,\phi) \in \DC_G$. There is an action of $G$ on $\Hom(M,B_w)$, analogous to \eqref{Gxaction}, which sends $f \co M \to B_w$ to the
composition \[\psi_{B_w}^{-1} \circ g(f) \circ \phi \co M \to gM \to gB_w \to B_w. \]

\begin{claim} \label{wtcharmap} $P_{(M,\phi), x}$ is the trace of $\si$ acting on the graded vector space $\Hom(M,B_w)$ modulo $\DC_{< w}$. In other words, define the \emph{weighted
character map} as the map $[\DC]_\si \to \HC(W,S,L)$, which sends $[M,\phi]$ to $\sum_{x \in W} P_{(M,\phi),x} H_x$ where $P$ is the trace just defined. Then it is the inverse
isomorphism to the isomorphism defined by sending the Kazhdan-Lusztig generator $c_s \in \HC(W,S,L)$ to $[B_s,+1]$. \end{claim}

If Claim \ref{wtcharmap} is true then, assuming that $[B_w] = b_w$ under the isomorphism $[\DC] \cong \HC(W')$, it is easy to deduce that $[B_w,+1] = c_w$ under the isomorphism
$[\DC]_\si \cong \HC(W,L)$. After all, we need only show that $[B_w, +1]$ is self-dual, and that the polynomials $P_{(B_w,+1),x}$ are concentrated in strictly positive degrees (except
when $x=w$). The self-duality follows immediately from the self-duality of $B_w$ (see also the discussion of duality in \S\ref{subsec-indecomposable}). The morphism spaces
$\Hom(B_w,B_x)$ are known to be concentrated in strictly positive degrees by the Soergel Hom formula, and taking the trace under any element will not change this fact.

We sketch a proof of Claim \ref{wtcharmap}, which imitates the non-equivariant proof (using light leaves) that the character map from $[\DC]$ to $\HC(W')$ is the inverse isomorphism.
Both sides (the weighted character map and the weighted Grothendieck group) are compatible with direct sum decompositions and with twisting by a character in $G^*$. Using this, we can
reduce to the case where $(M,\phi)$ is a tensor product of the generating objects $(B_s,+1)$ for $s \in S$. The weighted character map clearly sends $[\IM,+1]$ to $1$. One need only
check an induction step: that the weighted character map intertwines the tensor product by $(B_s,+1)$ and multiplication by $c_s$.

Let us temporarily assume that the Coxeter embedding is standard, so that $B_s$ is actually an object $\un{s}$ in $\DC$ instead of its Karoubi envelope (i.e. it is a sequence of distant
reflections). Then the induction step follows from the ``light leaf philosophy.'' Essentially, the claim is that, for an arbitrary sequence $\un{x}$ in $\DC$, $\Hom(M \ot \un{x}, B_w)$ has a basis modulo $\DC_{< w}$ given by diagrams of the following form.

\[{
\labellist
\small\hair 2pt
 \pinlabel {$M$} [ ] at 31 5
 \pinlabel {$z$} [ ] at 32 41
 \pinlabel {$w$} [ ] at 59 77
 \pinlabel {$\un{x}$} [ ] at 101 4
\endlabellist
\centering
\ig{1}{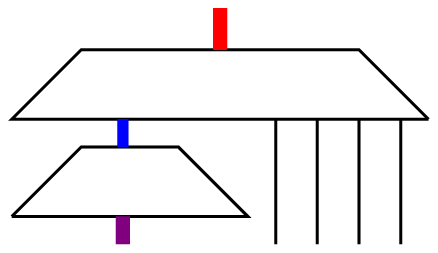}
}\]

\noindent Here, the bottom trapezoid is a map in a basis for $\Hom(M,B_z)$ modulo $\DC_{< z}$, for some other element $z$. Then, one chooses a subexpression $\eb$ of $\un{x}$,
expressing an element $y$ with $zy = w$. The light leaf algorithm states how to construct the top trapezoid, a map from $B_z \ot \un{x} \to B_w$. This algorithm works one simple
reflection at a time. For the first simple reflection in $\un{x}$, the subexpression $\eb$ has determined whether we take it or leave it, and the element $z$ determines whether that
reflection would go up or down in the Bruhat order. For each of the four possibilities, there is a diagram which is the first tier of the trapezoid. This handles the special case where
$\un{x}$ has length $1$; the general case just iterates this procedure.

Now suppose that $\un{x}$ is the sequence of commuting simple reflections expressing $s \in S$. The subexpression $\eb$ of $\un{x}$ is determined by a subset $J \subset I(s)$,
specifying which simple reflections are included and which excluded. If $z \in W$ is $G$-invariant, then $z$ is either a minimal or a maximal coset representative for the parabolic
subgroup determined by $I(s)$. Thus, for each simple reflection in $\un{x}$, whether it goes up or down is determined by $z$, not by the subset $J$. In particular, the action of $G$ on
the upper trapezoids is the same as the action of $G$ on the formal span of subsets $J \subset I(s)$.

In general, the space of morphisms $\Hom(M \un{x}, B_w)$ modulo $\DC_{< w}$ is the direct sum, over all $z$, of the tensor product of $\Hom(M, B_z)$ with the formal span of subsets $J
\subset I(s)$ for which $zy = w$. The action of $G$ will permute the various $z$ and the various $J$ accordingly. As can easily be seen, a large portion of this representation is
automatically traceless for $\si$. Only the terms where $z \in W$ and $w \in W$, and $J = \emptyset$ or $J = I(s)$ will contribute to the trace. These two terms, the ``multiplication by
$1$" and the ``multiplication by $s$" terms, agree with the two terms that appear when you multiply a standard basis element by $c_s$.

When the Coxeter embedding is not standard, a similar argument can be made. There are \emph{generalized light leaves} (due to the author and Williamson) which describe morphisms between
tensor products of longest elements of parabolic subgroups, not just tensor products of simple reflections, and they follow an analogous light leaf philosophy, constructed by a tiered
algorithm. They have not yet appeared in the literature, though hopefully this will be remedied soon. The proof using generalized light leaves is very similar, only instead of
considering the action of $G$ on subset $J \subset I(s)$, one considers the action of $G$ on the parabolic subgroup generated by $I(s)$. This is very similar to the sketched argument
made in \S\ref{subsec-quadraticII}.

\bibliographystyle{plain}
\bibliography{mastercopy}

\end{document}

%% file: H3left.tex
\begin{array}{c}
\begin{tikzpicture}[scale=0.2]
\def\hmax{13}

\def\aa{30}
\def\ab{90}
\def\ac{150}
\def\ada{-30}
\def\adb{-90}
\def\adc{-150}

\def\ha{18}
\def\hb{54}
\def\hc{90}
\def\hd{126}
\def\he{162}
\def\hda{-18}
\def\hdb{-54}
\def\hdc{-90}
\def\hdd{-126}
\def\hde{-162}

\begin{scope}[yscale=1.4]

\foreach \x in {1,2,...,15} \coordinate (t\x) at
(\x,\hmax);
\foreach \x in {1,2,...,15} \coordinate (b\x) at
(\x,0);

\foreach \c/\x/\y in {
1/13/2,
2/10/3,
3/7/4,
4/4/5,
5/9/5,
6/7/6,
7/4/7,
8/11/6,
9/9/7,
10/7/7.5,
11/11/8,
12/9/9,
13/6/10,
14/13/9,
15/11/10,
16/8/11
}
\node[circle,draw,inner sep = 0mm, minimum size=0mm] (c\c) at (\x,\y)
{};

\path (c1) edge [out=\hda,in=90] (b15)
                edge [out=\hdb,in=90,blue]  (b14)
                edge [out=\hdc,in=90]  (b13)
                edge [out=\hdd,in=90,blue]  (b12)
                edge [out=\hde,in=90]  (b11)
                edge [out=\ha,in=\ada,blue] (c14)
                edge [out=\hb,in=\hda,black]  (c11)
                edge [out=\hc,in=\ada,blue]  (c8)
                edge [out=\hd,in=\hda,black]  (c5)
                edge [out=\he,in=\ada,blue]  (c2);

\path (c2) edge [out=\adb,in=90,red] (b10)
                edge [out=\adc,in=90,blue]  (b9)
                edge [out=\aa,in=\adb,red]  (c8)
                edge [out=\ab,in=\hdb,blue]  (c5)
                edge [out=\ac,in=\ada,red]  (c3);

\path (c3) edge [out=\adb,in=90,blue] (b7)
                edge [out=\adc,in=90,red]  (b5)
                edge [out=\aa,in=\hdd,blue]  (c5)
                edge [out=\ab,in=\adb,red]  (c6)
                edge [out=\ac,in=\hda,blue]  (c4);

\path (c4) edge [out=\hdb,in=90,black]  (b6)
                edge [out=\hdc,in=90,blue]  (b4)
                edge [out=\hdd,in=90]  (b3)
                edge [out=\hde,in=90,blue]  (b2)
                edge [out=\ha,in=\hde] (c5)
                edge [out=\hb,in=\adc,blue]  (c6)
                edge [out=\hc,in=\hdd]  (c10)
                edge [out=\hd,in=\adb,blue]  (c7)
                edge [out=\he,in=-90]  (t1);

\path (c5) edge [out=\hdc,in=90] (b8)
                edge [out=\ha,in=\adc,blue] (c8)
                edge [out=\hb,in=\hdc]  (c11)
                edge [out=\hc,in=\adb,blue]  (c9)
                edge [out=\hd,in=\hdb]  (c10)
                edge [out=\he,in=\ada,blue]  (c6);

\path (c6) edge [out=\aa,in=\adc,red]  (c9)
                edge [out=\ab,in=\hdc,blue]  (c10)
                edge [out=\ac,in=\ada,red]  (c7);

\path (c7) edge [out=\adc,in=90,red]  (b1)
                edge [out=\aa,in=\hde,blue]  (c10)
                edge [out=\ab,in=\adc,red]  (c13)
                edge [out=\ac,in=-90,blue]  (t2);

\path (c8) edge [out=\aa,in=\adb,red]  (c14)
                edge [out=\ab,in=\hdb,blue]  (c11)
                edge [out=\ac,in=\ada,red]  (c9);

\path (c9) edge [out=\aa,in=\hdd,blue]  (c11)
                edge [out=\ab,in=\adb,red]  (c12)
                edge [out=\ac,in=\hda,blue]  (c10);

\path (c10) edge [out=\ha,in=\hde] (c11)
                edge [out=\hb,in=\adc,blue]  (c12)
                edge [out=\hc,in=\hdd]  (c16)
                edge [out=\hd,in=\adb,blue]  (c13)
                edge [out=\he,in=-90]  (t3);

\path (c11) edge [out=\ha,in=\adc,blue] (c14)
                edge [out=\hb,in=\hdc]  (t13)
                edge [out=\hc,in=\adb,blue]  (c15)
                edge [out=\hd,in=\hdb]  (c16)
                edge [out=\he,in=\ada,blue]  (c12);

\path (c12) edge [out=\aa,in=\adc,red]  (c15)
                edge [out=\ab,in=\hdc,blue]  (c16)
                edge [out=\ac,in=\ada,red]  (c13);

\path (c13) edge [out=\aa,in=\hde,blue]  (c16)
                edge [out=\ab,in=-90,red]  (t6)
                edge [out=\ac,in=-90,blue]  (t4);

\path (c14) edge [out=\aa,in=-90,red]  (t15)
                edge [out=\ab,in=-90,blue]  (t14)
                edge [out=\ac,in=\ada,red]  (c15);

\path (c15) edge [out=\aa,in=-90,blue]  (t12)
                edge [out=\ab,in=-90,red]  (t11)
                edge [out=\ac,in=\hda,blue]  (c16);

\path (c16) edge [out=\ha,in=-90,black] (t10)
                edge [out=\hb,in=-90,blue]  (t9)
                edge [out=\hc,in=-90,black]  (t8)
                edge [out=\hd,in=-90,blue]  (t7)
                edge [out=\he,in=-90,black]  (t5);

\end{scope}
\end{tikzpicture}
\end{array}

%% file: H3right.tex
\begin{array}{c}
\begin{tikzpicture}[scale=0.2]
\def\hmax{13}

\def\adc{30}
\def\adb{90}
\def\ada{150}
\def\ac{-30}
\def\ab{-90}
\def\aa{-150}

\def\hde{18}
\def\hdd{54}
\def\hdc{90}
\def\hdb{126}
\def\hda{162}
\def\he{-18}
\def\hd{-54}
\def\hc{-90}
\def\hb{-126}
\def\ha{-162}

\begin{scope}[yscale=1.4]
\begin{scope}[rotate=180]

\foreach \x in {1,2,...,15} \coordinate (t\x) at
(\x,\hmax);
\foreach \x in {1,2,...,15} \coordinate (b\x) at
(\x,0);

\foreach \c/\x/\y in {
1/13/2,
2/10/3,
3/7/4,
4/4/5,
5/9/5,
6/7/6,
7/4/7,
8/11/6,
9/9/7,
10/7/7.5,
11/11/8,
12/9/9,
13/6/10,
14/13/9,
15/11/10,
16/8/11
}
\node[circle,draw,inner sep = 0mm, minimum size=0mm] (c\c) at (\x,\y)
{};
\end{scope}

\path (c1) edge [out=\hda,in=-90] (b15)
                edge [out=\hdb,in=-90,blue]  (b14)
                edge [out=\hdc,in=-90]  (b13)
                edge [out=\hdd,in=-90,blue]  (b12)
                edge [out=\hde,in=-90]  (b11)
                edge [out=\ha,in=\ada,blue] (c14)
                edge [out=\hb,in=\hda]  (c11)
                edge [out=\hc,in=\ada,blue]  (c8)
                edge [out=\hd,in=\hda]  (c5)
                edge [out=\he,in=\ada,blue]  (c2)
;

\path (c2) edge [out=\adb,in=-90,red] (b10)
                edge [out=\adc,in=-90,blue]  (b9)
                edge [out=\aa,in=\adb,red]  (c8)
                edge [out=\ab,in=\hdb,blue]  (c5)
                edge [out=\ac,in=\ada,red]  (c3);

\path (c3) edge [out=\adb,in=-90,blue] (b7)
                edge [out=\adc,in=-90,red]  (b5)
                edge [out=\aa,in=\hdd,blue]  (c5)
                edge [out=\ab,in=\adb,red]  (c6)
                edge [out=\ac,in=\hda,blue]  (c4);

\path (c4) edge [out=\hdb,in=-90,black]  (b6)
                edge [out=\hdc,in=-90,blue]  (b4)
                edge [out=\hdd,in=-90]  (b3)
                edge [out=\hde,in=-90,blue]  (b2)
                edge [out=\ha,in=\hde] (c5)
                edge [out=\hb,in=\adc,blue]  (c6)
                edge [out=\hc,in=\hdd]  (c10)
                edge [out=\hd,in=\adb,blue]  (c7)
                edge [out=\he,in=90]  (t1);

\path (c5) edge [out=\hdc,in=-90] (b8)
                edge [out=\ha,in=\adc,blue] (c8)
                edge [out=\hb,in=\hdc]  (c11)
                edge [out=\hc,in=\adb,blue]  (c9)
                edge [out=\hd,in=\hdb]  (c10)
                edge [out=\he,in=\ada,blue]  (c6);

\path (c6) edge [out=\aa,in=\adc,red]  (c9)
                edge [out=\ab,in=\hdc,blue]  (c10)
                edge [out=\ac,in=\ada,red]  (c7);

\path (c7) edge [out=\adc,in=-90,red]  (b1)
                edge [out=\aa,in=\hde,blue]  (c10)
                edge [out=\ab,in=\adc,red]  (c13)
                edge [out=\ac,in=90,blue]  (t2);

\path (c8) edge [out=\aa,in=\adb,red]  (c14)
                edge [out=\ab,in=\hdb,blue]  (c11)
                edge [out=\ac,in=\ada,red]  (c9);

\path (c9) edge [out=\aa,in=\hdd,blue]  (c11)
                edge [out=\ab,in=\adb,red]  (c12)
                edge [out=\ac,in=\hda,blue]  (c10);

\path (c10) edge [out=\ha,in=\hde] (c11)
                edge [out=\hb,in=\adc,blue]  (c12)
                edge [out=\hc,in=\hdd]  (c16)
                edge [out=\hd,in=\adb,blue]  (c13)
                edge [out=\he,in=90]  (t3);

\path (c11) edge [out=\ha,in=\adc,blue] (c14)
                edge [out=\hb,in=\hdc]  (t13)
                edge [out=\hc,in=\adb,blue]  (c15)
                edge [out=\hd,in=\hdb]  (c16)
                edge [out=\he,in=\ada,blue]  (c12);

\path (c12) edge [out=\aa,in=\adc,red]  (c15)
                edge [out=\ab,in=\hdc,blue]  (c16)
                edge [out=\ac,in=\ada,red]  (c13);

\path (c13) edge [out=\aa,in=\hde,blue]  (c16)
                edge [out=\ab,in=90,red]  (t6)
                edge [out=\ac,in=90,blue]  (t4);

\path (c14) edge [out=\aa,in=90,red]  (t15)
                edge [out=\ab,in=90,blue]  (t14)
                edge [out=\ac,in=\ada,red]  (c15);

\path (c15) edge [out=\aa,in=90,blue]  (t12)
                edge [out=\ab,in=90,red]  (t11)
                edge [out=\ac,in=\hda,blue]  (c16);

\path (c16) edge [out=\ha,in=90,black] (t10)
                edge [out=\hb,in=90,blue]  (t9)
                edge [out=\hc,in=90,black]  (t8)
                edge [out=\hd,in=90,blue]  (t7)
                edge [out=\he,in=90,black]  (t5);

\end{scope}
\end{tikzpicture}
\end{array}